\documentclass[12pt,a4paper]{article}
\usepackage{amsmath,amsfonts,amsthm,amssymb,latexsym}
\usepackage{xspace}
\usepackage{lipsum}
\usepackage{scrextend}
\usepackage{enumitem}
\usepackage{comment}
\usepackage{xcolor}
\usepackage{graphicx}
\date{13th Dec}

\newcommand{\Z}{{\mathbb Z}}

\newcommand{\emptyword}{\epsilon}

\def\disting{{\mathbf d}}
\def\leftnbr{{\mathbf d}^L}
\def\rightnbr{{\mathbf d}^R}
\newcommand{\vv}{p}
\newcommand{\xx}{x} 
\newcommand{\uu}{r} 
\newcommand{\ee}{E} 
\newcommand{\ii}{I}
\newcommand{\jj}{J}
\newcommand{\kk}{K}
\renewcommand{\ll}{L}
\newcommand{\mm}{M}
\newcommand{\nn}{N}
\newcommand{\cR}{\mathcal{R}}
\newcommand{\Rev}{\mathsf{R}}
\newcommand{\fail}{{\bf{fail}}}
\renewcommand{\stop}{{\bf{stop}}}

\newcommand{\false}{\texttt{false}}
\def \p#1{{\rm pref}[#1]}
\def \s#1{{\rm suf}[#1]}
\def \f#1{{\rm f}[#1]}
\def \l#1{{\rm l}[#1]}

\newtheorem{theorem}{Theorem}[section] 
\newtheorem{lemma}[theorem]{Lemma} 
\newtheorem{proposition}[theorem]{Proposition}
\newtheorem*{proposition*}{Proposition}
\newtheorem{corollary}[theorem]{Corollary}
\theoremstyle{definition}
\newtheorem{definition}[theorem]{Definition}
\newtheorem{example}[theorem]{Example}

\newtheorem{procedure}[theorem]{Procedure}

\newenvironment{mylist}{\begin{list}{}{
\setlength{\parskip}{0mm}
\setlength{\topsep}{2mm}
\setlength{\parsep}{0mm}
\setlength{\itemsep}{0.5mm}
\setlength{\labelwidth}{7mm}
\setlength{\labelsep}{3mm}
\setlength{\itemindent}{0mm}
\setlength{\leftmargin}{12mm}
\setlength{\listparindent}{6mm}
}}{\end{list}}
\newenvironment{mylistb}{\begin{list}{}{
\setlength{\parskip}{0mm}
\setlength{\topsep}{2mm}
\setlength{\parsep}{0mm}
\setlength{\itemsep}{0.5mm}
\setlength{\labelwidth}{9mm}
\setlength{\labelsep}{3mm}
\setlength{\itemindent}{0mm}
\setlength{\leftmargin}{12mm}
\setlength{\listparindent}{6mm}
}}{\end{list}}
\newenvironment{mylistc}{\begin{list}{}{
\setlength{\parskip}{0mm}
\setlength{\topsep}{2mm}
\setlength{\parsep}{0mm}
\setlength{\itemsep}{0.5mm}
\setlength{\labelwidth}{1mm}
\setlength{\labelsep}{3mm}
\setlength{\itemindent}{0mm}
\setlength{\leftmargin}{4mm}
\setlength{\listparindent}{6mm}
}}{\end{list}}
\newenvironment{myproclist}{\begin{list}{}{
\setlength{\parskip}{0mm}
\setlength{\topsep}{0mm}
\setlength{\parsep}{0mm}
\setlength{\itemsep}{2mm}
\setlength{\labelwidth}{0mm}
\setlength{\labelsep}{0mm}
\setlength{\itemindent}{0mm}
\setlength{\leftmargin}{4mm}
\setlength{\listparindent}{3mm}
}}{\end{list}}

\title{Rewriting in Artin groups without $A_3$ or $B_3$ subdiagrams}

\author{Rub\'en Blasco-Garc\'ia, Mar\'ia Cumplido, Derek F. Holt,\\
Rose Morris-Wright and Sarah Rees}

\begin{document}
\maketitle

\begin{abstract}
We prove that the word problem in an Artin group $G$ based on a diagram without
$A_3$ or $B_3$ subdiagrams can be solved using a system  
of length
preserving rewrite rules which, together with free reduction, can be used to
reduce any word over the standard generators of $G$ to a geodesic word in $G$
 in quadratic time. This result builds on work of Holt and Rees, and of
Blasco, Cumplido and Morris-Wright. Those articles prove the same result for all Artin
groups that are either sufficiently large or 3-free, respectively.
\end{abstract}

{\bf Dedication:}

{\bf Acknowledgements:}
Mar\'ia Cumplido was supported by the research project PID2022-138719NA-I00, 
financed by MCIN/AEI/10.13039/501100011033/FEDER, UE, and by a Ram\'on y Cajal 2021 grant, 
also financed by the Spanish Ministry of Science and Innovation.
Sarah Rees was supported by Leverhulme foundation grant no.RPG-2022-025 during the work on this article.

\section{Introduction}
\label{sec:intro}
Our main result in this article is the following.
\begin{theorem}
	\label{thm:main}
Let $G$ be an Artin group defined over its standard generating set $S$ for
which the associated Coxeter diagram contains no $A_3$ or $B_3$ subdiagram;
	in other words, no subset $\{x,y,z\}$ of the standard generating set for $G$ satisfies the
relations $xyx=yxy$, $xz=zx$, $(y,z)_n = (z,y)_n$ with $n=3$ or 4.
Then there is a system of length preserving rewrite rules which, together
with free reduction, can be used to reduce any word over $S$ to a geodesic word
in $G$, in quadratic time.

Furthermore, any two geodesic words representing the same element of $G$ are
related by a sequence of rules from the system, each such involving just two
generators. 
\end{theorem}

The notation $(y,z)_n$ used in the statement of Theorem~\ref{thm:main}
denotes an alternating string $\cdots yzyz$ of length $n$ ending with $z$, 
as explained in Section~\ref{sec:notation}.
Throughout this article, $G$ will be as defined in the statement of
Theorem~\ref{thm:main}.

This result, together with the main results of \cite{HR12,HR13,BCMW} that it extends, addresses the general open question of whether all Artin groups have soluble word problem, and if so, whether the word problem can always be solved using a combination of length preserving rewrites together with free reduction.

Example~\ref{eg:n_atleast5} within Section~\ref{sec:RRS}
shows that the condition that excludes $B_3$ subdiagrams
(corresponding to generators $x,y,z$ with $xyx=yxy$, $xz=zx$,
$(y,z)_4 = (z,y)_4$) is necessary for
our algorithm 
that reduces words to geodesics to work.

The length preserving rules 
referred to 
in the statement of Theorem~\ref{thm:main} generalise the $\tau$-moves
of~\cite{HR12,HR13,BCMW}. 
Each rule has the effect of replacing a \emph{critical subword} $u$ of the
input word by another word $\tau(u)$ of the same length, whose first and last
letters both differ from the corresponding letters within $u$. When a sequence
of such moves leads to a reduction of the input word, it does so by moving from
left to right within the word, replacing successive overlapping subwords, until
the letter at the right-hand end of the final replacement subword is the inverse
of the subsequent letter in the word, and so a free cancellation is possible. 

In the large and sufficiently large groups studied in \cite{HR12,HR13},
the word problem can be solved as just described using critical subwords, which
in this case
are defined to be 2-generator words with particular constraints. In \cite{BCMW} the
same type of solution was found to work with a more general definition of a
critical word that allowed it to be \emph{pseudo 2-generator} but not
necessarily 2-generator. 
In this article we define \emph{pseudo 3-generator critical
words}, and prove that, if these are used alongside the others, then the same
type of solution works in our group $G$.
Our arguments in this article are extensions of the arguments
used to prove the main result of \cite{BCMW}: in the case where $G$ is 3-free (that is, its Coxeter diagram contains no edges labelled $3$) 
no $\tau$-moves on pseudo 3-generator critical words are necessary
in our rewrite system;
in that case  
the rewrite system is exactly as described in \cite{BCMW}
and our arguments reduce to the arguments of that article.

Our article is structured as follows. After this introduction and a section
explaining our notation, we define critical words of three different types and
the $\tau$-moves that deal with them in Section~\ref{sec:critical_tau},
recalling the definitions of critical words that are 2-generator or pseudo
2-generator from \cite{HR12} and \cite{BCMW}, and then providing a definition
for pseudo 3-generator critical words. We observe some basic properties of these words
and their associated $\tau$-moves. 
Then in Section~\ref{sec:RRS} we define \emph{rightward reducing sequences} of
$\tau$-moves. We define a set $W$ to consist of all words that do not admit
such a sequence, and describe a procedure, Procedure~\ref{proc:unique_optRRS},
that can be used to reduce any word to a representative of the same element
within $W$ in quadratic time; its correctness is proved in
Proposition~\ref{prop:unique_optRRS}.
Finally Section~\ref{sec:proofs} contains the proof of Theorem~\ref{thm:main}.
That section starts with the statement of Theorem~\ref{thm:main_details}, which 
restates the result of Theorem~\ref{thm:main},
providing more detail; in particular the set $W$ is proved to be the set of all
geodesic representatives of the elements of $G$, and the rewrite system 
of
Theorem~\ref{thm:main} is found to be the set of all rightward reducing
sequences.  Technical details needed in the proof of
Theorem~\ref{thm:main_details} are provided in
Propositions~\ref{lem:6.1}--\ref{lem:6.4}, stated and proved after the proof 
of the theorem. 

We end this section with an example, which indicates why we needed to generalise
the work of \cite{BCMW} in order to deal with the groups $G$ of
Theorem~\ref{thm:main}.

\begin{example}
\label{eg:tricky_w_in_G}
Consider the word
$w:=(b,c)_{n-1}\cdot aba \cdot {}_{n-2}(c,b)x^{-1}$ in the rank 3 Artin group
\[ \langle a,b,c \mid aba=bab,\,ac=ca,\, (b,c)_n = (c,b)_n \rangle, \]
where 
${}_{n-2}(c,b)$ denotes the alternating string $cbc\cdots$ of length $n-2$ that
starts with $c$ (see Section~\ref{sec:notation}), and $x$ is equal to $c$ if
$n$ is even or to $b$ if $n$ is odd.

Since we have the relation $aba= bab$,
the word $w$ represents the same element of the group as
$w':=(b,c)_{n-1}\cdot bab\cdot {}_{n-2}(c,b)x^{-1}=
 (c,b)_n\cdot a\cdot {}_{n-1}(b,c)x^{-1}$,
and we can rewrite $w'$ applying the following relations (which are particular cases of the $\tau$-moves we define in Section~\ref{sec:critical_tau})
\[ (c,b)_n \rightarrow (b,c)_n,\,ca\rightarrow ac,\,{}_n(c,b)
  \rightarrow {}_n(b,c), \]
through which we have
\begin{eqnarray*}
	w' &\rightarrow& (b,c)_n\cdot a \cdot {}_{n-1}(b,c)x^{-1} \\
	&\rightarrow& (c,b)_{n-1} \cdot ac \cdot {}_{n-1}(b,c)x^{-1}
	= (c,b)_{n-1}\cdot a \cdot {}_n(c,b)x^{-1}\\
	&\rightarrow& (c,b)_{n-1}\cdot a \cdot {}_n(b,c)x^{-1},
\end{eqnarray*}
to derive from $w'$ a word ending with  the subword $xx^{-1}$. There is
certainly a sequence of 2-generator $\tau$-moves followed by a free
cancellation that reduces $w$ (via $w'$)
to a shorter word representing the same element;
but it is not a sequence of moves that moves 
rightward within the word. 
There exists no such sequence, even using the 
more general $\tau$-moves of \cite{BCMW}.
We note that the proofs in \cite{HR12,HR13,BCMW} that the rewrite
systems of those articles reduce a word  to the empty word $\emptyword$ 
if and only if it represents the identity rely on the fact that $\tau$-moves
are applied in what we will call rightward reducing sequences. 

In fact, the prefix $(b,c)_{n-1}\cdot aba$ of $w$ is a pseudo 3-generator 
critical subword in the sense of this article, and the move that 
replaces this by $(c,b)_{n-1}\cdot acb$ is the first step of a rightward moving
sequence of the more general $\tau$-moves of this article that leads to
reduction of $w$, as we shall see in Example~\ref{eg:tricky_w_in_G_2}
\end{example}

This article was motivated by the work of \cite{BCMW}, 
which provided a first generalisation of the methods of \cite{HR12,HR13}.  Our
proof is inspired by the proof of that article, and although this article can be read without
reference to the arguments of \cite{BCMW}, we have deliberately chosen notation
that is as close to the notation of \cite{BCMW} as was possible,  and will
refer to parts of that article at times, in order to aid comparison.

\section{Notation}\label{sec:notation}

We define $S$ to be the set of standard generators of the Artin group
$G$, and $A:= S \cup S^{-1}$, the set of standard generators and their inverses.

For $X \subseteq A$, we defined a \emph{word over $X$} to be a string of
elements from $X \cup X^{-1}$; the elements of that string are called the 
\emph{letters} of $w$. 
The \emph{length} of a word $w$ is its length as a string. 
A subword of $w$ that is a power of a letter, and is maximal as such, is called
a \emph{syllable} of $w$;
the \emph{syllable length} of $w$ is the number of syllables within it
(for example, the string $a^2b^{-3}ab^4$ is a word over $\{a,b\}$ of syllable
length 4).
Elements of $S$ are called \emph{positive} letters,
and elements of $S^{-1}:= \{x^{-1}: x \in  S\}$ negative letters;
a word is called positive if it involves only positive letters,
negative if it involves only negative letters, and otherwise is called
\emph{unsigned}.
We define the \emph{name} of a letter $x$ to be the generator within $\{x,x^{-1}\}$ to which it corresponds.
A word is called \emph{geodesic in $G$} if it has minimal length as a
representative of the element of $G$ that it represents.

For $x,y,z \in A$ with distinct names, we refer to words over $\{x,y\}$
as \emph{ $\{x,y\}$-words}, also as \emph{$2$-generator words},
and we refer to words over $\{x,y,z\}$
as \emph{ $\{x,y,z\}$-words}, also as \emph{3-generator words}.

For $x,y \in A$, as above, we define ${}_{n}(x,y)$ to be
the alternating string $xyxy \cdots$ of length $n$ beginning with $x$ and
$(y,x)_n$
to be the alternating string $\cdots yxyx$ of length $n$ ending with $x$
If $x,y$ are both positive, such a string is called \emph{positive alternating}
of length $n$, while if $x,y$ are both negative, it is called
\emph{negative alternating} of length $n$.

For any word $w$, we denote by $\f{w}$ and $\l{w}$ its first and last letter,
and by $\p{w}$ and $\s{w}$ its maximal proper prefix and maximal proper suffix.
Note that $w=\f{w}\s{w}=\p{w}\l{w}$.

\section{Critical words and $\tau$-moves}

\label{sec:critical_tau}
The concept of a critical word as a particular type of 2-generator word in an
Artin group was introduced in \cite{HR12},
and in that article and \cite{HR13} those words were an essential part of a
process for reduction in large and sufficiently large Artin groups. 
In this article we shall need further critical words, on three or more generators; these
will be of two kinds: pseudo 2-generator words (introduced in \cite{BCMW}) and
pseudo 3-generator critical words,
both defined in Section~\ref{sec:more_gens} below.

\subsection{Critical words and $\tau$-moves involving just two generators}
\label{sec:2gen}

Critical 2-generator words are geodesic in 2-generated Artin groups 
\cite{MairesseMatheus} and vital to the recognition of and reduction to geodesics in those Artin groups that are studied in \cite{HR12,HR13,BCMW}.
We recall the definition and some technical results from \cite{HR12,HR13} below,
as well as some further technical lemmas that we shall need in this article;
the proofs could easily be omitted from a first reading.

Let $X=\{x,y\}$ be a subset of $S$ with $x \neq y$, and let $m_{xy}$ be the
integer associated with $x,y$ in the presentation of $G$.
The element $\Delta$
of the subgroup $G_{xy}=\langle x,  y\,|\, {}_{m_{xy}}(x,y) ={}_{m_{xy}}(y,x) \rangle$ of $G$ that is represented by the 
alternating product $(x,y)_{m_{xy}}$ 
is known as the \emph{Garside element} of $G_{xy}$ and is 
also represented by the other alternating product $(y,x)_{m_{xy}}$.
Let $v$ be a word over $X$.
Following \cite{HR12} we  define $p(v)$ to be the minimum of $m_{xy}$ and the 
maximal length of a positive alternating subword of $v$,
and
$n(v)$ to be the minimum of $m_{xy}$ and the 
maximal length of a negative alternating subword of $v$. 

\begin{definition}[2-generator critical words]
\label{gen:2gen_critical}
The 2-generator word $u$ over $\{x,y\}$ is defined to be
\emph{($\{x,y\}$-)critical} if
\begin{mylist}
\item[(i)] $p(u)+n(u)=m_{xy}$; and
\item[(ii)] either 
\begin{mylist}
\item[(a)] $u$ is positive of the form ${}_{m_{xy}}(x,y)\xi$
or $\xi(x,y)_{m_{xy}}$ and $u$ has a unique subword $v$ with $p(v)=m_{xy}$;
or \item[(b)]
$u$ is negative of the form ${}_{m_{xy}}(x^{-1},y^{-1})\xi$ or
$\xi(x^{-1},y^{-1})_{m_{xy}}$
and $u$ has a unique subword $v$ with $n(v)=m_{xy}$;
or \item[(c)] $u$ is unsigned, and has one of the two forms
${}_{p(u)}(x,y)\xi (z^{-1},t^{-1})_{n(u)}$
or 
${}_{n(u)}(x^{-1},y^{-1})\xi (z,t)_{p(u)}$,
where $\{z,t\}=\{x,y\}$.
\end{mylist}
\end{mylist}
	Notice that $\xi$ above denotes any (positive, negative or unsigned) word over $\{x,y\}.$ \

\end{definition}
We remark that a critical word must have syllable length at least $m_{xy}$.

\begin{lemma}\label{lem:critlintest}
We can check in linear time whether a $2$-generator word is critical.
\end{lemma}
\begin{proof} For any $w$, we can calculate $p(w)$ and $n(w)$. We can also 
test in a single scan of a signed word whether it is critical.
\end{proof}

\begin{lemma}\label{lem:critsubword}
Suppose that the critical $2$-generator word $u$ has a critical subword $v$,
and so $u = u_pvu_s$ for some words $u_p$, $u_s$. Then the word $vu_s$ has a
critical suffix.
\end{lemma}
\begin{proof}
Since $p(u)+n(u) = p(v)+n(v) = m_{xy}$ with $p(v) \le p(u)$ and $n(v) \le n(u)$,
we must have $p(u)=p(v)=p(vu_s)$ and $n(u)=n(v)= n(vu_s)$, and the result
follows easily. 
\end{proof}
 
Still following \cite{HR12}, we define an involution $\tau$ on
the set of all
$\{x,y\}$-critical words as follows.

\begin{definition}[2-generator $\tau$-moves]
	\label{def:2gen_tau}
	First we define $\delta$ to be the permutation of $\{x,y\}$ 
	for which, for each $a \in \{x,y\}$,
	$\delta(a)$ is equal in $G_{xy}$ to $\Delta a \Delta^{-1}$;
	$\delta$ is the identity permutation if $m_{xy}$ is even, has order 2 
	if $m_{xy}$ is odd, and extends naturally to a length preserving permutation
	of the set of all words over $\{x,y\}$, with $\delta(a^{-1}) = \delta(a)^{-1}$,
	and $\delta(uv)$ the concatenation of $\delta(u)$ and $\delta(v)$.

Now, for unsigned 2-generator critical words, we define $\tau$ by
\begin{eqnarray*}
\tau({}_p(x,y)\,\xi\,(z^{-1},t^{-1})_n)
&:=& {}_n(y^{-1},x^{-1})\,\delta(\xi )\,(t,z)_p,\\
\tau({}_n(x^{-1},y^{-1})\,\xi\,(z,t)_p)
&:=& {}_p(y,x)\,\delta(\xi )\,(t^{-1},z^{-1})_n.
\end{eqnarray*}

Then, for positive and negative 2-generator critical words, we define $\tau$ as
follows, where $\xi$ is assumed non-empty in the final four equations.
\begin{eqnarray*}
	\tau({}_{m_{xy}}(x,y))&:=& {}_{m_{xy}}(y,x),\\
	\tau({}_{m_{xy}}(x^{-1},y^{-1}))&:=& {}_{m_{xy}}(y^{-1},x^{-1})\\
	\tau({}_{m_{xy}}(x,y)\,\xi) &:=& \delta(\xi)\,(z,t)_{m_{xy}},
\quad\hbox{\rm where}\quad z=\l{\xi},\,\{x,y\}=\{z,t\},\\
	\tau(\xi\,(x,y)_{m_{xy}}) &:=& {}_{m_{xy}}(t,z)\,\delta(\xi),
\quad\hbox{\rm where}\quad z=\f{\xi},\,\{x,y\}=\{z,t\},\\
	\tau({}_{m_{xy}}(x^{-1},y^{-1})\,\xi) &:=& \delta(\xi)\,(z^{-1},t^{-1})_{m_{xy}},
\quad\hbox{\rm where}\quad z=\l{\xi}^{-1},\,\{x,y\}=\{z,t\},\\
	\tau(\xi\,(x^{-1},y^{-1})_{m_{xy}}) &:=& {}_{m_{xy}}(t^{-1},z^{-1})\,\delta(\xi),
\quad\hbox{\rm where}\quad z=\f{\xi}^{-1},\,\{x,y\}=\{z,t\}.
\end{eqnarray*}
\end{definition}

It is verified in \cite{HR12} that a 2-generator critical word $u$ and its
image $\tau(u)$ represent the same element of the subgroup
$G_{xy}$ of $G$; hence $u=_G \tau(u)$. Replacement of a critical
subword $u$ of a (probably) longer word by $\tau(u)$ is called a $\tau$-move.

In  this article we shall often refer to the $\tau$-moves we have just defined
on 2-generator critical words as \emph{2-generator $\tau$-moves} in order to
distinguish them from more general $\tau$-moves defined in \cite{BCMW} for
pseudo 2-generator (P2G) words and in this article for (other) words on more than two generators.

\begin{lemma}\label{lem:2gengeo}
Let $w$ and $w'$ be geodesic words over the standard generators of a
$2$-generator Artin group that represent the same element of the group. Then we
can transform $w$ to $w'$ using a sequence of 2-generator $\tau$-moves.
\end{lemma}
\begin{proof}  By \cite[Theorem 2.4]{HR12} we can transform each of $w$ and $w'$
to their (equal) shortlex normal forms using a sequence of (2-generator)
$\tau$-moves. Since the reverse of a 2-generator $\tau$-move is also a 
2-generator $\tau$-move, the result follows.
\end{proof}

We observe the following facts about 2-generator critical words and
$\tau$-moves, which we shall need later. 
\begin{lemma}
	\label{lem:2gen_taufacts}
For any 2-generator critical word $u$, the names of the first letters of $u$
and $\tau(u)$ are distinct, as are the names of the last letters of
$u$ and $\tau(u)$.
\end{lemma}
\begin{proof}
This proved as part of \cite[Proposition 2.1]{HR12}
\end{proof}

The following lemmas will be used in connection with P3G-critical words,
which will be defined in the next section. Lemma~\ref{lem:abcrit2} will be
used in the proof that our rewriting procedure for solving the word problem
in the Artin group $G$ runs in quadratic time.

\begin{lemma}\label{lem:abcrit1}
Suppose that $m_{ab}= 3$, 
let $v$ be an $\{a,b\}$-word such that $\f{v}$ and $\l{v}$ both have name $a$,
and suppose that $v$ can be transformed using a sequence of $\tau$-moves, each applied to a critical subword of $v$, to a
word $v':= b^{i}a^{j}b^{k}$ with $i,j,k$ nonzero integers.
      Then $v$ is a critical word, and $v' = \tau(v)$.

Furthermore, if $v$ is a positive or a negative word then $|j|=1$, and for $\epsilon = \pm 1$ either
\begin{mylist}
\item[(i)] $|i|=1$ and
    $v = a^{\epsilon k'} b^{\epsilon} a^{\epsilon}$, $v'=b^{\epsilon}a^{\epsilon}b^{\epsilon k'}$
     with $k'>0$; or
\item[(ii)] $|k|=1$ and
$v = a^{\epsilon} b^{\epsilon} a^{\epsilon i'}$, $v'=b^{\epsilon i'}a^{\epsilon}b^{\epsilon}$
   with $i'>0$.
\end{mylist}

On the other hand, if $v$ is an unsigned word then for $\epsilon=\pm 1$ either 
\begin{mylist}
\item[(i)] $|i|=1$ and
$v = a^{\epsilon}b^{-\epsilon(j'-1)}a^{\epsilon(k'-1)}b^{-\epsilon} a^{-\epsilon}$, $v'=
    b^{-\epsilon} a^{-\epsilon j'}b^{\epsilon k'}$ with $j',k'>0$; or
\item[(ii)] $|k|=1$ and
  $v = a^{\epsilon}b^{\epsilon}a^{-\epsilon(i'-1)}b^{\epsilon(j'-1)}a^{-\epsilon}$,
$v'=b^{-\epsilon i'}a^{\epsilon j'}b^{\epsilon}$ with $i',j'>0$.
  $v = a^{\epsilon}b^{\epsilon}a^{-\epsilon(i'-1)}b^{\epsilon(j'-1)}a^{-\epsilon}$,
     $v'=b^{-\epsilon i'}a^{\epsilon j'}b^{\epsilon}$ with $i',j'>0$.
---
\end{mylist}
\end{lemma}
\begin{proof} By \cite[Proposition 4.3]{MairesseMatheus}, a word $w$ in the
	$2$-generator Artin group $G_{ab}:= \langle a,b \mid aba=bab \rangle$ is geodesic if and
only if $p(w) + n(w) \le 3$, and there is a second geodesic word
defining the same group element if and only if $p(w)+n(w)=3$.
Since the word $v':= b^{i}a^{j}b^{k}$ satisfies $p(w)+n(w) \le 3$ for all values of
$i,j,k$, this word is geodesic and hence so is $v$, and since $v \ne w$ we must
have $p(v)+n(v) = 3$ with $p(v) = p(b^{i}a^{j}b^{k})$ and
$n(v) = n(b^{i}a^{j}b^{k})$.

Suppose first that $v$ is a positive word, so $p(v)=3$ and $n(v)=0$. The
case when $v$ is negative is similar. So $\f{v}=\l{v}=a$.
Since $p(b^{i}a^{j}b^{k}) = 3$, we must have $j=1$ and $i,k \ge 1$.
We claim that the only positive words that are equal
in $G_{ab}$ to $b^{i}a^{j}b^{k}=b^i a b^k$ are the words of the form
$b^{i'}aba^{i-i'}b^{k-1}$
for $0 \le i' \le i$ and $b^{i-1} a^{k-k'}bab^{k'}$ for $0 \le k' \le k$.
This claim can be verified by checking that any substitution of $aba$ by $bab$
or vice versa results in the replacement of one word of this form by another.
It follows from the claim, together with $\f{v} = \l{v}=a$,
that either $k=1$ and $i'=0$ and hence $v'=b^iab$ and $v=aba^i$,
or $i=1$ and $k'=0$ and hence $v'=bab^k$ and $v=a^kba$.
In both these cases we see that $v,v'$ are critical, and $\tau(v)=v'$. 

Otherwise $v$ is unsigned, so either $p(v)=2$ and $n(v)=1$, or $p(v)=1$ and
$n(v)=2$. Suppose that $p(v)=2$ and $n(v)=1$; the other case is similar.
Since the same applies to the word $b^{i}a^{j}b^{k}$ we must have either
$i,j>0$ and $k<0$, or $i<0$ and $j,k>0$.
Assume the former---again the other
case is similar. Then the word $b^{-1}v$ is not geodesic in $G_{ab}$ so,
using the results of~\cite{MairesseMatheus} again, we must have $p(b^{-1}v) =
n(b^{-1}v) = 2$, so $\f{v}=a^{-1}$. Similarly, since the word $vb$ is
non-geodesic, we must have $p(vb)=3$ and $n(vb)=1$, so $v$ has
the suffix $ba$, and hence $v$ is critical as claimed.

Still assuming that $v$ has the prefix $a^{-1}$ and the suffix $ba$ (the
other three cases are similar), we have $v = a^{-1} \xi ba =_{G_{ab}}
ba \delta(\xi) b^{-1}$ for some word $\xi$, and so
$b^{i-1} a^{j}b^{k+1} =_{G_{ab}} a \delta(\xi)$ with both words geodesics.
Now an unsigned geodesic word with first letter $a$ cannot have a
geodesic representative with first letter $b$ (because any $\tau$-move
changing this letter $a$ would change it to $b^{-1}$), so we must have
$i=1$ and, since $a^{j}b^{k+1}$ is the unique geodesic representative
of its group element, we have $a^{j-1}b^{k+1}=\delta(\xi)$,
$\xi = b^{j-1}a^{k+1}$. 
So now $v = a^{-1}b^{j-1}a^{k+1}ba$, with $j>0,k<0$, and
and for all possible $j,k$ we see that $v$ is critical, and
$\tau(v)= ba^jb^k=v'$.
\end{proof}

\begin{lemma}\label{lem:abcrit2}
Suppose that $m_{ab} = 3$, and let $v$ be an $\{a,b\}$-word such that $\f{v}$
and $\l{v}$ both have name $a$.
Then we can decide in linear time whether $v$ can be transformed using a
sequence of $\tau$-moves to a word $b^ia^jb^k$ with $i,j,k$ nonzero integers
and, if so, compute the resulting transformed word.
\end{lemma}
\begin{proof}
We can calculate $p(v)$ and $n(v)$ in linear time and if $p(v) + n(v) \ne 3$
then no such transformation is possible, and we return \false.
So we can assume that $p(v)+n(v)=3$ and hence that $v$ is a geodesic word.
If $v$ is a positive word then, by Lemma~\ref{lem:abcrit1}, the
transformation in question exists if and only if either $j=k=1$ and $v=a^kba$,
or $i=j=1$ and $v = aba^i$, which we can test directly in linear time.
The case when $v$ is a negative word is similar.

Similarly, if $v$ is unsigned, then the transformation in question exists if
and only if $v$ has one of the forms defined in Lemma~\ref{lem:abcrit1},
which can again be tested in linear time.
\end{proof}

\subsection{Critical words and $\tau$-moves involving more than two generators}
\label{sec:more_gens}
In addition to the 2-generator critical words defined in Section~\ref{sec:2gen},
we shall need two kinds of critical words involving more than two generators,
both words that are \emph{pseudo 2-generator} as defined in \cite{BCMW} and
some further words that we shall call \emph{pseudo 3-generator critical words},
and define below.
For clarity, we repeat the definition of pseudo 2-generator critical words from
\cite[Definition 3.1]{BCMW} as Definition~\ref{def:P2G} below.

\begin{definition}[Pseudo 2-generator critical words of type $\{a,b\}$]
\label{def:P2G}
Let $a, b \in S$ with $2 < m_{ab} < \infty$ and denote by
$P$ the set $\{a, b\}$. Let $u \in A^*$ be a word such that
$\f{u}, \l{u} \in P \cup P^{-1}$. Let $u_p$ be the prefix of $u$ up to
but not including the first instance of a letter in $P \cup P^{-1}$
with a different name from $\f{u}$.
Similarly, let $u_s$ be the longest suffix of $u$ that starts after the last letter of $u_p$ and does not contain a
letter in $P \cup P^{-1}$ with a name different from $\l{u}$.
Factor $u$ as
$u = u_p u_q u_s$.

We say that $u$ is a \emph{pseudo 2-generator (P2G) word} of type $\{a,b\}$
if all of the letters in $u_p$ commute with $\f{u}$, all of the
letters in $u_q$ not in $P \cup P^{-1}$ commute with both $a$ and $b$,
and all of the letters in $u_s$ commute with $\l{u}$. 

We call the letters of  $P=\{a,b\}$ 
the \emph{pseudo-generators} of $u$
and the letters of $u$ not in
$P \cup P^{-1}$ its \emph{internal} letters.
We define 
$\hat{u}$ to be the word obtained from $u$ by deleting all of the
internal letters.

A P2G word $u$ over $\{a,b\}$ is said to be
\emph{($\{a,b\}$-)critical} if $\hat{u}$ is a critical
$2$-generator word over $\{a,b\}$.
\end{definition}

\begin{definition}[Subwords $\alpha(u),\beta(u),\rho(u)$ of a
P2G-critical word]\label{def:P2Gabr}

We now define some further notation associated with a given P2G word $u$.
If $u$ is a P2G word, as above, then it 
can be transformed using commuting relations to a word
$\alpha \rho \hat{u} \beta$, where $\hat{u}$ is as defined in
Definition~\ref{def:P2G} and $\alpha:= \alpha(u)$,
$\beta:= \beta(u)$, $\rho := \rho(u)$ are defined as follows.
\begin{mylistb}
\item[$\alpha(u)$] is the word obtained from $u_p$ by deleting from it all
letters in $P \cup P^{-1}$.

By definition of a P2G word, each letter in $\alpha$ must commute with
$\f{u}$ and may or may not commute with the other pseudo-generator.
\item[$\beta(u)$] is the word obtained from $u_s$ by deleting from it all
letters in $P \cup P^{-1}$
as well as all letters in it that can be pushed past the leftmost letter
$\f{\hat{u}}$ of $\hat{u}$ within $u$ via commutations. 
By definition of a P2G word, each letter in $\beta$ must commute with $\l{u}$
and may or may not commute with the other pseudo-generator.
The word $\beta$ cannot contain a letter that commutes with both pseudo-generators if
that letter can be pushed past the first letter of $\hat{u}$, possibly
as part of a block of similar such letters; such letters belong in $\rho$.

In particular, $\f{\beta}$ cannot compute with both pseudo-generators,
and nor can any letter of $\beta$ that already commutes with all letters to
the left of it within $\beta$.
\item[$\rho(u)$] is the word obtained from $u$ by deleting all letters in $P \cup P^{-1}$
and all of the letters in $\alpha$ and $\beta$, and is a subword of $u_qu_s$.
It follows from the definitions of $\alpha$ and $\beta$ that each
letter in $\rho$ must commute with both pseudo-generators.
\end{mylistb}

The word $\alpha \rho \hat{u} \beta$ can be obtained from $w$ by pushing as
many internal letters as possible to the left via commuting relations, and then
pushing the remaining internal letters to the right via commuting relations.

\end{definition}

The following definition of $\tau$-moves for  critical P2G words is that of
\cite{BCMW}, which extends the definition of \cite{HR12,HR13} for 2-generator
words.

\begin{definition}[P2G $\tau$-moves] \label{def:P2G_tau}
Let $u$ be an $\{a,b\}$-critical P2G word.  Let
$\alpha(u),\rho(u)$, and $\beta(u)$ and the 2-generator critical word
$\hat{u}$ be as in Definition~\ref{def:P2G}, and $\tau(\hat{u})$ as in
Definition~\ref{def:2gen_tau}.
We define
\[\tau(u) := \alpha(u)\rho(u)\tau(\widehat{u})\beta(u),\] 
We shall refer to the replacement by $\tau(u)$ of a P2G-critical word $u$
a \emph{$\tau$-move}, \emph{of type $\{a,b\}$}.
\end{definition}

\begin{example} \label{eg:P2G_critical_tau}
Let $a,b,c \in S$ with $m_{ab}=3$, $m_{bc}=4$ and $m_{ac}=2$.

Then the words $ab^5a^{-1}$ and $bc^2bcb^{-1}c^{-1}$ are 2-generator critical
words with $\tau(ab^5a^{-1})=b^{-1}a^5b$ and
$\tau(bc^2bcb^{-1}c^{-1})=c^{-1}b^{-1}cbc^2b$.

The word $u_1:= acb^5c^{-1}a^{-1}$ is P2G-critical of type $\{a,b\}$,
with $\widehat{u_1} = ab^5 a^{-1}$, $\alpha(u_1)=c$, $\beta(u_1)=c^{-1}$,
$\rho(u_1) = \emptyword$, and $\tau(u_1) = cb^{-1}a^5bc^{-1}$.

The word $u_2:= bc^2bcb^{-1}a^2c^{-1}$ is P2G-critical of type $\{b,c\}$,
with $\widehat{u_2} = bc^2bcb^{-1}c^{-1}$, $\alpha(u_2)=\rho(u_2) =\emptyword$,
$\beta(u_2)=a^2$, and $\tau(u_2) = c^{-1}b^{-1}cbc^2ba^2$.
\end{example}

On the subset of P2G words that are actually 2-generator, the definitions of
critical words and of $\tau$-moves coincide with those already given for
2-generator words.

We make the following observations about P2G-critical words in $G$, and
their associated $\tau$-moves.
\begin{lemma} \label{lem:P2G_taufacts}
Let $u$ be a P2G-critical word, with pseudo-generators $\{a,b\}$.
Then
\begin{mylist}
\item[(i)] the first and last letters of $u$ are in $\{a,a^{-1},b,b^{-1}\}$,
and equal to the first and last letters of $\hat{u}$,
\item[(ii)] the names of the first letters of $u$ and $\tau(u)$ are distinct,
as are the names of the last letters of $u$ and $\tau(u)$,
\item[(iii)] $\f{\tau(u)} \in \{a,b\} \iff \alpha(u) = \rho(u) = \emptyword$,
\item[(iv)] $\l{\tau(u)} \in \{a,b\} \iff \beta(u) = \emptyword$,
\end{mylist}
\end{lemma}
\begin{proof}
This can be observed from the definition together with the properties observed
in Lemma~\ref{lem:2gen_taufacts}.
\end{proof}

Note that the words $\alpha(u)$, $\rho(u)$ and $\beta(u)$ introduce a lack of
symmetry into the definition of a P2G-critical word $u$ and of $\tau(u)$, and if any of
those three words is non-empty then $\tau(u)$ is not itself critical.

The word $w=(b,c)_{n-1}\cdot aba \cdot {}_{n-2}(c,b)x^{-1}$ considered in
Example~\ref{eg:tricky_w_in_G} for our selected group $G$ still does not admit
a \emph{rightward reducing sequence} (RRS) as defined in \cite{BCMW}.
To remedy this, we extend our definition of a critical
word to allow not only P2G words but also some further words over $S$ to be
considered critical, each of which is constructed out of two P2G words,
one of type $\{a,b\}$ and the other of type $\{b,c\}$, as explained below.
Then we define $\tau$-moves on these new critical words, with the
intention that, by using these new $\tau$-moves together with the ones
from \cite{BCMW} defined above, we can find rightward 
reducing sequences to reduce any non-geodesic word over $S$.

Note that the condition that $m_{bc} \geq 5$ is built into Definition~\ref{def:P3G}; it is because of this constraint that our hypotheses on $G$ exclude
$A_3$ and $B_3$ subdiagrams from the Coxeter diagram that defines it.

\begin{definition}[Pseudo 3-generator words and critical words, of type $(a,b,c)$]
\label{def:P3G}
Let $a,b,c \in S$ satisfy $m_{ab}=3,m_{ac}=2,m_{bc}\geq 5$.
We call a word $u \in A^*$ with $\f{u} \in \{b,b^{-1},c,c^{-1}\}$ and
$\l{u} \in \{a,a^{-1}\}$ \emph{pseudo 3-generator (P3G)  
of type $(a,b,c)$} if it has the form
$u = u_p u_q u_r$, where:
\begin{mylistc}
\item $\f{u_p}$ has name $b$ or $c$, and $u_p$ is the prefix of $u$ up to but
not including the first instance of a letter in $u$ that does not commute with
$\f{u}$; 
\item $\f{u_q}$ has name in $\{b,c\}$ distinct from that of $\f{u_p}$, and
every letter of $u_q$ with name distinct from $b$ and $c$ commutes with both
$b$ and $c$; 
\item $u_r$ is a P2G word of type $\{a,b\}$, whose first and last letters have
name $a$.
\end{mylistc}
Further, we say that the word is \emph{pseudo 3-generator (P3G) critical 
	of type $(a,b,c)$} (abbreviated as $(a,b,c)$-critical) if additionally:
\begin{mylist}
\item[(i)] $u_r$ is P2G-critical, and
the $\{a,b\}$-critical word
$\widehat{u_r}$ that is formed by deleting from $u_r$ all letters other than $a,b$ and their inverses  is transformed by $\tau$
         into a
word of the form $b^{\ii} a^{\jj} b^{\kk}$ for some nonzero
$\ii,\jj,\kk \in \Z$.
\item[(ii)] $u^\# := u_p u_q \alpha(u_r)\rho(u_r) b^{\ii}$ is a P2G-critical
word of type $\{b,c\}$ with $\beta(u^\#)$ trivial.
\end{mylist}
We call the elements $\{a,b,c\}$ the \emph{pseudo-generators} 
of $u$, and call a letter of $u$ \emph{internal} if its name is not in the set
	$\{a,b,c\}$. 

We call the elements of the subset $P:=\{b,c\}$ of $\{a,b,c\}$ the
\emph{primary pseudo-generators of $u$}, and observe that these are
the pseudo-generators of $u^\#$.

We note that our definitions here of $u_p,u_q$ correspond to those of $u_p,u_q$ 
in a P2G-word. But the suffix $u_r$ of a P3G word plays a different role 
	to that of the suffix $u_s$ of a P2G word.
\end{definition}

	We note also that:
\begin{mylist}
\item[(1)]
        Following condition (i) on $u_r$, the critical 2-generator words $\widehat{u_r},\tau(\widehat{u_r})$ must match $v,v'$ as described in the statement of Lemma~\ref{lem:abcrit1}.
\item[(2)]
$u^\#$ is a prefix of a word that can be derived from $u$ by a
sequence of commutation rules followed by a $\tau$-move on the
2-generator word $\widehat{u_r}$.
\end{mylist}

\begin{definition}[Subwords $\alpha(u),\beta(u),\rho(u)$ of a P3G-critical
word]\label{def:P3Gabr}
For $u$ P3G-critical, we define $\alpha:=\alpha(u)$, $\beta:=\beta(u)$ and
$\rho:=\rho(u)$ as follows.
\begin{mylist}
\item[$\alpha(u)$] $:= \alpha(u^\#)$ is the word obtained from $u_p$ by
deleting all instances of $\f{u}$ and $\f{u}^{-1}$.
So, each letter of $\alpha(u)$ has name not in $\{b,c\}$, must commute with
$\f{u}$ (which has name $b$ or $c$) and may or may not commute with the other 
pseudo-generator in $\{b,c\}$.  If $\f{u}$ has name $c$, then it is possible
for $\alpha(u)$ to contain $a$ or $a^{-1}$.
\item[$\beta(u)$] $:= \beta(u_r)$. So each letter in $\beta(u)$ has name not in
$\{a,b\}$, and must commute with $a$, but not necessarily with $b$. The word
$\beta(u)$ cannot contain a letter that commutes with both $a$ and $b$ if that
letter can be pushed past the first letter of $\hat{u_r}$ using commutations,
possibly as part of a block of similar such letters; such letters belong in
$\rho(u)$. It is possible for $\beta(u)$ to contain $c$.

In particular, $\f{\beta(u)}$ cannot commute with both $a$ and $b$, and nor can
any letter of $\beta(u)$ that commutes with all letters to the left of it
within $\beta(u)$.
\item[$\rho(u)$] $:= \rho(u^\#)$. So each letter in $\rho(u)$
has name not in $\{b,c\}$, and must commute with both $b$ and $c$.

Note that $\rho(u)$ must contain $\rho(u_r)$ and all letters in $\alpha(u_r)$
except for any with name $c$, so all letters in $\alpha(u_r)$ must commute
with $b$ and $c$ as well as with $a$, and all letters in
$\rho(u_r)$ must commute with $c$ as well as with $a$ and $b$.
\end{mylist}
\end{definition}

Applying first commutation rules within $u_r$, then the relation
$\widehat{u_r} =_G b^{\ii}a^{\jj}b^{\kk}$, and then commutation rules
within $u^\#$, we see that
\begin{eqnarray*}
u &=& u_p u_q u_r =_G u_pu_q \alpha(u_r) \rho(u_r) \widehat{u_r} \beta(u_r)\\
&=_G& u_p u_q \alpha(u_r)\rho(u_r)  b^{\ii} a^{\jj} b^{\kk} \beta(u_r) 
= u^\# a^{\jj} b^{\kk} \beta(u_r) \\
&=_G& \alpha(u^\#)\rho(u^\#) \widehat{u^\#} a^{\jj} b^{\kk} \beta(u_r),
\end{eqnarray*}
where $\widehat{u^\#}$ is the $\{b,c\}$-word formed by deleting all letters other than $b,c$ and their inverses from the P2G-word $u^\#$.
We have $\tau(\widehat{u^\#})$  equal to $\p{\tau(\widehat{u^\#})}c^\ee$,
where $\ee = \pm 1$.
So now,  applying $\tau$ to $\widehat{u^\#}$,
followed by more commutation rules, we derive 
\begin{align*}
u &=_G \alpha(u^\#)\rho(u^\#) \tau(\widehat{u^\#}) a^{\jj} b^{\kk} \beta(u_r) 
= \alpha(u^\#)\rho(u^\#)  \p{\tau(\widehat{u^\#})} c^{\ee} a^{\jj} b^{\kk} \beta(u_r)\\
&=_G \alpha(u^\#)\rho(u^\#)  \p{\tau(\widehat{u^\#})} a^{\jj} c^{\ee} b^{\kk} \beta(u_r).
&&\mbox{\rm (1)}
\end{align*}

\begin{definition}[$\tau$-moves of type $(a,b,c)$]
\label{def:P3G_tau}
In the notation defined above, for the P3G-critical word
$u = u_p u_q u_r$ of type $(a,b,c)$, we define
\[ \tau(u) :=
\alpha(u) \rho(u) \p{\tau(\widehat{u^\#})} a^{\jj} c^{\ee} b^{\kk} \beta(u), \] 
	where $\ee=\pm 1$ as defined by
Equation 1 above.
We call these new moves \emph{$\tau$-moves of type $(a,b,c)$}.
\end{definition}
Note that in order to verify the identity $\tau(u)=_G u$ we have applied $\tau$-moves to
2-generator subwords moving from right to left within the word,
rather than from left to right.
This is why we need to collect this set of moves together
into a single $\tau$-move on the P3G-critical word.

\begin{example} \label{eg:P3G}
The subword $u:= (b,c)_{n-1}\cdot aba$ of $w$ from Example~\ref{eg:tricky_w_in_G}
	is $(a,b,c)$-critical,
with ${u}_p$ a single generator, 
${u}_q = (b,c)_{n-2}$, and ${u}_r = aba=\widehat{{u}_r}=_G bab.$
Then $u^\#=(b,c)_{n-1}b$, where $\alpha(u)$ and $\beta(u)$ are both empty, and
the integers $\ii,\jj,\kk$ of Definition~\ref{def:P3G}(iii) are all equal to 1.

Further $\tau(u^\#)=(b,c)_n=(c,b)_{n-1}c$, so $\ee=1$.
It follows that $\tau(u) = (c,b)_{n-1}acb$.
\end{example}

As in \cite{HR12,HR13,BCMW},
we shall reduce words in our group $G$ using particular sequences of 
overlapping $\tau$-moves, which (again as in \cite{HR12,HR13,BCMW}) we shall
call \emph{rightward reducing sequences}.

\section{Rightward reducing sequences of $\tau$-moves}
\label{sec:RRS}

In order to reduce a word $w$ using a sequence of $\tau$-moves (including the
new moves defined in Section~\ref{sec:critical_tau}),
we shall write $w$ as a concatenation of subwords, to which we associate
a sequence of words in $A^*$ that we shall call a
\emph{rightward reducing sequence}, abbreviated as RRS.
Our definition of an RRS extends the definition of \cite{BCMW},
which was already based on a definition within \cite{HR12,HR13}.
We use notation similar to that of \cite{BCMW} to facilitate comparison.

An RRS will be a sequence $U=u_1,\ldots,u_m,u_{m+1}$ of words in $A^*$,
with $u_i$ critical for $i \leq m$, which is associated with a factorisation of 
a word $w$ that is to be reduced, as we now describe.

\begin{definition}[RRS] \label{def:RRS}
Let
$U=u_1,\ldots,u_m,u_{m+1}$ be a sequence of words over $A$,
for which each $u_i$ with $1 \le i \le m$ is P2G-critical or P3G-critical
and $u_{m+1}$ is a non-empty word whose first letter commutes with all other
letters in $u_{m+1}$.
For $1 \le i \le m$,  write $\alpha_i,\rho_i$ and
$\beta_i$ for $\alpha(u_i),\rho(u_i)$ and $\beta(u_i)$ respectively.

For a word $w$ over $A$, we say that \emph{$w$ admits $U$ as an RRS of
length $m$} if $w$ can be written as a concatenation of words
$w=\mu w_1 \cdots w_mw_{m+1}\gamma$, with $w_i$ 
non-empty for $1 \le i \le m$, $u_1=w_1$, $\f{u_{m+1}} = \f{\gamma}^{-1}$ and,
for each $i$ with $2 \le i \le m+1$, $u_i$ is determined
(in terms of $u_{i-1}$ and $w_i$) according to case (i) or (ii) below.
\begin{mylist}
\item[(i)] if $u_{i-1}$ is P2G-critical,  
then $u_i = \l{\tau(\widehat{u_{i-1}})} \beta_{i-1} w_i$;
\item[(ii)] if $u_{i-1}$ is P3G-critical of type $(a,b,c)$, then
$u_i = c^{\ee} b^{\kk}  \beta_{i-1} w_i$,
\end{mylist}
where $\tau$ is as defined in Definition~\ref{def:P2G_tau},
$\ee$ and $\kk$ are as defined in
Definitions~\ref{def:P3G} and~\ref{def:P3G_tau},
and $\widehat{u_{i-1}}$ denotes the 2-generator word over two pseudo-generators
formed by deleting all occurrences of internal generators.

We call the factorisation $\mu w_1 \cdots w_mw_{m+1}\gamma$ of $w$
the \emph{decomposition associated with $U$}.

For each $i$ we define $P_i \subset S$ to be the set of pseudo-generators  
of $u_i$ when $u_i$ is P2G-critical, and the set of primary pseudo-generators
of $u_i$ when $u_i$ is P3G-critical.
\end{definition}
As in \cite{BCMW}, we can think of an RRS as a way of applying successive
$\tau$-moves and commutations, moving from left to right, until a free
reduction becomes possible.
To apply the RRS $U$ to $w$, we first replace the subword $u_1=w_1$ by
$\tau(u_1)$ (or by $\s{w_1}\f{w_1}$ when $m=0$). 
For $m>0$, this results in
a word with subword $u_2$ consisting of a suffix of $\tau(u_1)$ followed
by $w_2$. When $m>1$ we replace $u_2$ by $\tau(u_2)$ and continue like
this, 
replacing $u_i$ by $\tau(u_i)$ for each $i\leq m$, 
finally replacing $u_{m+1}$ by $\s{u_{m+1}}\f{u_{m+1}}$. The 
resulting word $w'$ has subword $\f{\gamma}^{-1}\f{\gamma}$,
and hence can be freely reduced. We write $w \rightarrow_U w'$.

Note that the word $w_{m+1}$ (but not $u_{m+1}$) 
can be empty if $m>0$. If $m=0$ then
$w_{m+1}=w_1$ is non-empty, with first letter $\f{\gamma}^{-1}$.
So a word that is not freely reduced
admits an RRS of length 0, 
with $u_1=w_1=\f{\gamma}^{-1}$.

Note also that if $\gamma$ has length $>1$ then 
$\p{w}$ must also admit an RRS of length $m$, associated with its
factorisation $\mu w_1 \cdots w_mw_{m+1}\p{\gamma}$; so if $\p{w}$ does not admit an RRS, then $\gamma$ must be a single letter.

\begin{lemma}
\label{lem:RRSdetails}
Let $U=u_1,\ldots,u_m,u_{m+1}$ with $m>0$ be an RRS for a word $w$ over $A$
with associated decomposition $w=\mu w_1 \cdots w_mw_{m+1}\gamma$. 
For $1 \le i \le m$, define $\alpha_i := \alpha(u_i)$,
$\beta_i := \beta(u_i)$, and $\rho_i := \rho(u_i)$.  Then
\begin{mylist}
\item[(i)] $u_m$ cannot be P3G-critical, and $\beta_m = \emptyword$.
\item[(ii)] 
	Suppose that $u_{i-1}$ is
P2G-critical for some $1<i<m+1$.
	Then $P_i \cap P_{i-1} \neq \emptyset$.

	Now suppose that $\beta_{i-1} \ne \emptyword$, and that $P_{i-1} = \{a,b\}$,
	where $a$ is the name of $\l{u_{i-1}}$.
	Then $\alpha_i = \emptyword$ and $P_i=\{b,c\}$ for some $c \neq a$;
	hence $|P_i \cap P_{i-1}| = 1$.
Furthermore,
\begin{mylist}
	\item[either] $\beta_{i-1}$ is a non-trivial power of $c$;
	\item[or] 
		$\beta_{i-1} = c^{\ll}d^{\mm}$ and $\beta_i=d^{\mm}$
		for some generator $d$ with $m_{bd}=2$, $m_{cd}\ge 5$,
where $\ll,\mm \ne 0$,
		and $u_i$ is P2G-critical with $m_{bc}=3$.
\end{mylist}
\item[(iii)] 
Suppose that $u_{i-1}$ is $(a,b,c)$-critical for some $1<i<m+1$. Then
	$\alpha_i = \emptyword$, and $P_i = \{b,c\}=P_{i-1}$.
Furthermore,
\begin{mylist}
\item[either] $\beta_{i-1}$ is a (possibly trivial) power of $c$;
\item[or] 
$\beta_{i-1} = c^\ll(a')^\mm$ for some $a'$ and $\ll, \mm \ne 0$,
and $u_i$ is P3G-critical of type $(a',c,b)$, 
where $m_{bc}=5$ and $m_{aa'}=2$.
\end{mylist}
\item[(iv)] Suppose that $u_i$ is P3G-critical of type $(a',b',c')$.
Then $w_i$ has a suffix $u$ that contains letters with names $c'$ and $b'$,
with the $c'$ occurring to the left of the $b'$.
Furthermore, 
\begin{mylist}
\item[either] $(u_i)_r$ is a suffix of $w_i$,
and $w_i$ has a suffix $vu$ where $\f{v}$ has name $b'$;
\item[or] $i>1$, $m_{b'c'}=5$ and $u_{i-1}$ is P3G-critical of
type $(a,c',b')$ for some generator $a$.
	\end{mylist}
In either case, the prefix $\vv_i:=\mu w_1\cdots w_i$ of $w$ has a suffix that is P2G-critical of type $\{a',b'\}$.
\end{mylist}
Note that it follows from parts (ii) and (iii) that $\beta_i$  has at most two
syllables for $i < m$. 
\end{lemma}
\begin{proof} (partly from \cite[Lemma 3.11]{BCMW}).

(i) If $u_m$ had type $(a,b,c)$ then $\f{u_{m+1}} = c^{\ee}$ would not
commute with the second letter of $u_{m+1}$ (which has name $b$), 
contrary to Definition~\ref{def:RRS}.

If $\beta_m \neq \emptyword$, then $\f{\beta_m}$ is a letter of $u_{m+1}$, and
	so must commute with $\f{u_{m+1}}=\l{\tau(\widehat{u_m})}$ 
But then, by our final
observation about $\beta_m=\beta(u_m)$ within Definition~\ref{def:P2Gabr},
$\f{\beta_m}$ would be in $\rho_m$, contrary to assumption.

(ii) If $\beta_{i-1} = \emptyword$ then $\l{\tau(u_{i-1})} =
\l{\tau(\widehat{u_{i-1}})}\in P_i \cap P_{i-1}$, which proves (ii).
Otherwise $\beta_{i-1} \ne \emptyword$ and we
define $c:= \f{\beta_{i-1}}$. 
Then $c$ commutes with $a$ by Definition~\ref{def:P2G}, and $c$ cannot
commute with $b$, since then it would be in $\rho_{i-1}$.
By Lemma~\ref{lem:2gen_taufacts}, $\l{\tau(\widehat{u_{i-1}})}$ has name $b$
and so, by Definition~\ref{def:RRS}, $u_i$ has prefix $b^{\pm 1}c^{\pm 1}$;
then it follows from Definitions~\ref{def:P2G} and ~\ref{def:P3G}
that $(u_i)_p = b^{\pm 1}$ and $c \in P_i$, so
$b$ and $c$ are pseudo-generators of $u_i$ and by Definitions~\ref{def:P2Gabr}
and~\ref{def:P3Gabr}  we have $\alpha_i = \emptyword$. 
If $u_i$ is P3G-critical with third pseudo-generator $a'$, then it has type
$(a',b,c)$ if $a'$ commutes with $c$, and type $(a',c,b)$ otherwise.
So $P_i = \{b,c\}$, and $|P_i \cap P_{i-1}| = 1$

If $\beta_{i-1}$ is a power of $c$ then we are done, so suppose not.
Then $\beta_{i-1}$ has a prefix $c^{\ll}d^{\pm 1}$, with $\ll \ne 0$ and
$d$ a generator with $d \neq c$.  Since $d$ is an internal generator of
$u_{i-1}$, we also have $d \ne a$ and $d \ne b$.
Since $d$ occurs in $\beta_{i-1}$, it must commute with $a$.
But it cannot commute with both $b$ and $c$, or it could be pushed by
commutation moves to the left in $u_{i-1}$, and then it would lie in
$\rho_{i-1}$ rather than $\beta_{i-1}$.
Note that $b^{\pm 1}c^{\ll}$ is a prefix of $u_i$.

We prove by contradiciton that $u_i$ is P2G-critical.
For suppose that $u_i$ is
P3G-critical. Then, using the notation of Definition~\ref{def:P3G}, 
we have $(u_i)_p = b^{\pm 1}$, $(u_i)_q = c^{\ll}$, and $d^{\pm 1}$ is the
first letter of $(u_i)_r$, so $d=a'$, and hence must commute with one of $b,c$.

Now if $d$ commutes with $c$ but not with $b$ then $u_i$ has type $(a',b,c)$
so $(u_i)_r$ is P2G-critical of type $\{a',b\}$,
and $(u_i)^\#$ (which is P2G-critical of type $\{b,c\}$) has the form
$(u_i)_p(u_i)_q \alpha((u_i)_r)\rho((u_i)_r)b^{\ii}$.
By definition, powers of $c$, but not of $b$, might be contained in the subword
$\alpha((u_i)_r)\rho((u_i)_r)$ of $(u_i)^\#$.
So the $\{b,c\}$-word $\widehat{(u_i)^\#}$ formed from $(u_i)^\#$ by deleting
all letters with names other than $b,c$ (Definition~\ref{def:P2G}) has the form
$b^{\pm 1}c^{L'}b^{\ii}$ for some $L'$ and has syllable length three.
Since, by Definition~\ref{def:P3G}, we have $m_{bc}\geq 5$, the word
$\widehat{(u_i)^\#}$ cannot be $\{b,c\}$-critical, which is a contradiction.

Similarly, if $d$ commutes with $b$ but not with $c$, then
$$(u_i)^\#=(u_i)_p(u_i)_q \alpha((u_i)_r)\rho((u_i)_r)c^{\ii}$$ is
P2G-critical of type $\{b,c\}$, and $(u_i)_r$ P2G-critical of type $\{a',c\}$,
where powers of $b$, but not of $c$, might be contained in the 
subword $\alpha((u_i)_r)\rho((u_i)_r)$ of $(u_i)^\#$.
But then the $\{b,c\}$-word $\widehat{(u_i)^\#}$ has syllable length at most
four, which again contradicts $m_{bc} \ge 5$.

So now we have established, by contradiction, that $u_i$ is P2G-critical,
and hence, by Definition~\ref{def:RRS}(i), we have $u_i =b^{\pm 1}\beta_{i-1}w_i$, where $\beta_{i-1}$ has prefix
$c^{\ll}d^{\pm 1}$. As already observed, we have $d \ne b$, and
$d$ cannot commute with both $b$ and $c$, so the letter $d^{\pm 1}$ cannot be
in $\rho_i = \rho(u_i)$, and hence it must be in $\beta_i=\beta(u_i)$. It
follows from the definition of $\beta(u_i)$ in Definition~\ref{def:P2Gabr} that
this letter $d^{\pm 1}$ together with everything following it within $u_i$ is
within the suffix $(u_i)_s$ of $u_i$, which (by definition) can contain only
one of the two pseudo-generators $\{b,c\}$ of $u_i$.
So $\widehat{u_i}$ has syllable length at most 3, and
$u_i$ can only be P2G-critical if $\l{u_i}=b^{\pm 1}$ and
$m_{bc}=3$, in which case, $\widehat{u_i}$ can be critical of the form
$b^{\pm 1}c^{\ll}b^{\mm}$.  Since $d$ occurs in  $\beta_i$, it must commute
with $\l{u_i}=b^{\pm 1}$, and so $m_{db}=2$. 
Since our hypotheses for $G$  exclude the possibility of $A_3$ and $B_3$
subdiagrams, we see also that $m_{cd}\geq 5$.
Further since $\beta_i \neq \emptyword$, it follows from part (i) that $i<m$.
	
Now, by the above arguments applied to $u_i$ in place of $u_{i-1}$,
we find that $u_{i+1}$ is either P2G-critical of type $\{c,d\}$, or
P3G-critical of type  $(a',c,d)$ or $(a',d,c)$ for some $a'$
and, since $m_{cd} \geq 5$, we have $\beta_i=d^{\mm}$ for some $\mm \ne 0$,
which implies $\beta_{i-1}=b^{\ll}d^{\mm}$. This completes the proof of (ii).

(iii) If $u_{i-1}$ has type $(a,b,c)$ then by Definition~\ref{def:RRS}\,(ii)
we have $u_i = c^{\ee} b^{\kk}  \beta_{i-1} w_i$ with
$\ee = \pm 1$ and $\kk \ne 0$; we recall that $b$ and $c$ do not commute.
Then it follows from Definitions~\ref{def:P2G} and~\ref{def:P3G} for P2G- and
P3G-critical words that $b$ and $c$ must be pseudo-generators for $u_i$ and
hence $\alpha_i=\emptyword$ (Definitions~\ref{def:P2Gabr},~\ref{def:P3Gabr}),
and we have completed the proof of the first part of the statement.

Now suppose that $\beta_{i-1} \neq \emptyword$.

Suppose first that $u_i$ is P2G-critical. Since (by definition) the word
$\beta_{i-1}$ consists of internal generators of the type-$\{a,b\}$
P2G word $(u_{i-1})_r$,
the word $\beta_{i-1}$ cannot contain a letter with name $b$.  Also, since
$\l{(u_i)_r}$ has name $a$, every letter in $\beta_{i-1}$ commutes
with $a$.  If $\f{\beta_{i-1}}$ were in $\beta_i$, then it would
be within $(u_i)_s$, which only contains one of the two pseudo-generators
$\{b,c\}$, and hence the $\{b,c\}$-word $\widehat{u_i}$ would have syllable
length at most 3, so less than $m_{bc}$, and could not be critical.
If $\f{\beta_{i-1}}$ were in $\rho_i$, then it would commute with $a$, $b$ and
$c$, and so it could be pushed to the left in $u_{i-1}$ and would be in
$\rho_{i-1}$ rather than $\beta_{i-1}$. So $\f{\beta_{i-1}}$ must have name $c$.
If $\beta_{i-1}$ is not a power of $c$ then it has a prefix $c^{\ll}d^{\pm 1}$,
where $d \neq b,c$.  By the same arguments that we applied to
$\f{\beta_{i-1}}$, this letter $d^{\pm 1}$ cannot be in $\beta_i$ or $\rho_i$;
so there is no such letter, and
$\beta_{i-1}$ must be a power of $c$.

It remains to consider the case when $u_i$ is P3G-critical of type
$(a',b,c)$ or $(a',c,b)$. Again $\beta_{i-1}$ cannot contain $b$,
and the letter $\f{\beta_{i-1}}$ cannot commute with both $b$ and $c$
(or by Definition~\ref{def:P3G} it would be in $\rho_{i-1}$), and so
(again by Definition~\ref{def:P3G}) that letter cannot be in 
$\rho_i=\rho(u_i^\#)$.

By Definition~\ref{def:RRS}(ii), $\f{\beta_{i-1}}$ is preceded in $u_i$ by
$c^Eb^K$. If $\f{\beta_{i-1}}$ has name other than $c$, then 
it follows from Definition~\ref{def:P3G} that $\beta_{i-1}$ is a prefix of
the suffix $(u_i)_r$ of $u_i$, with $\f{\beta_{i-1}}=a'^{\pm 1}$, 
and the subwords $(u_i)_p,(u_i)_q$ are the powers $c^E$ and $b^K$. 
Then the subword $\alpha((u_i)_r)\rho((u_i)_r)$ of $(u_i)^\#$ (see 
Definition~\ref{def:P3G}) can contain just one of $b$ and $c$ (whichever
one commutes with $a'$), and so the $\{b,c\}$-word $\widehat{(u_i)^\#}$ has
syllable length at most four (less than $m_{bc}$), and cannot be critical.

So $\f{\beta_{i-1}}$ must have name $c$.
Suppose that $\beta_{i-1}$ is not a power of $c$, and so it has a prefix
$c^{\ll}d^{\pm 1}$ for some generator $d \ne b,c$. It follows from 
Definition~\ref{def:P3G} that $d^{\pm 1}$ is the first letter of $(u_i)_r$, and 
so $d=a'$, and $d$ commutes with one of $b,c$.

Now if $u_i$ has type $(a',b,c)$ (that is, $d=a'$ commutes with $c$), then
by Definition~\ref{def:P3G}
$(u_i)^\#= (u_i)_p(u_i)_q \alpha((u_i)_r)\rho((u_i)_r)b^\ii$,
where (Definition~\ref{def:RRS}) $u_i := c^Eb^K\beta_{i-1}w_i$.
So $(u_i)_p = c^E, (u_i)_q = b^Kc^{\ll}$, and, since $(u_i)_r$ should be 
P2G-critical of type $\{a',b\}$ and hence $\alpha((u_i)_r)\rho((u_i)_r)$ might
contain letters with name $c$ but none with name $b$,
the $\{b,c\}$-word $\widehat{(u_i)^\#}$ has syllable length at most 4.
So this cannot happen.

However, if $u_i$ has type $(a',c,b)$ (that is, $d=a'$ commutes with $b$), then
by Definition~\ref{def:P3G}
$(u_i)^\#= (u_i)_p(u_i)_q \alpha((u_i)_r)\rho((u_i)_r)c^\ii$, where again
$u_i := c^Eb^K\beta_{i-1}w_i$. In this case it is possible that 
$\alpha((u_i)_r)$ contains letters with name $b$; since we know that
$\beta_{i-1}$ can contain no such letters, such a letter would be within $w_i$.
Then the word $\widehat{(u_i)^\#}$ would have a prefix $c^Eb^Kc^L$,
followed by a power of $b$ lying within $\alpha((u_i)_r)\rho((u_i)_r)$,
and a suffix $c^\ii$; so $\widehat{(u_i)^\#}$ would have syllable
length five, and could  be $\{b,c\}$-critical provided that $m_{bc}=5$.

In this case, if $\beta_{i-1}$ had a third syllable, it would have
a prefix $c^{\ll}(a')^{\mm}e^{\pm 1}$ for some generator $e$ with
$e \ne a'$ or $b$, and then $e^{\pm 1}$ would be in $\alpha((u_i)_r)$
(which we know is non-empty),
so it would commute with $a'$. Then it would be in $\rho(u_i^\#)$ and  hence
commute with $b$ and $c$. But since it is in $\beta_{i-1}$ it also commutes
with $a$, so it could be pushed to the left in $u_{i-1}$ and by
Definition~\ref{def:P3G} would lie in $\rho_{i-1}$ rather than $\beta_{i-1}$.
We conclude that $\beta_{i-1} = c^{\ll}(a')^\mm$ for some integer $\mm$.

(iv) This is clear when $i=1$ and $w_i=u_i$ is critical,
so we may assume that $i>1$.
We shall use the fact that the $\{b',c'\}$-word $\widehat{(u_i)^\#}$ is
critical, and so must have at least $m_{b'c'}\geq 5$ syllables,
the rightmost of which (by Definition~\ref{def:P3G}) is a power of $b'$.

If $u_{i-1}$ is P2G-critical then, since
$u_i = \l{\tau(\widehat{u_{i-1}})} \beta_{i-1} w_i$ where, by part (ii),
$\beta_{i-1}$ is a (possibly trivial) power of a generator, 
the letters of at least $m_{b'c'}-2\geq 3$ syllables of $\widehat{(u_i)^\#}$
must be within $w_i$. 
We can choose $u$ to contain the letters of the second  syllable of 
$\widehat{(u_i)^\#}$
from  the right, and then $v$, immediately to the left of $u$ and containing
some letters with name $b'$, is still within $w_i$.
Note that the prefix $\l{\tau(\widehat{u_{i-1}})} \beta_{i-1}$ of $u_i$ is
within $(u_i)_p(u_i)_q$,and so the P2G-critical word $(u_i)_r$ is a suffix
of $w_i$ and hence of $\vv_i$ in this case.

Otherwise $u_{i-1}$ is P3G-critical of type $(a,b,c)$ with $\{b,c\} =\{b',c'\}$,
and $u_i = c^{\ee} b^{\kk}  \beta_{i-1} w_i$, where 
$\beta_{i-1}$  is as described in (iii).
If $\beta_{i-1} = \emptyword$, then, as above, $w_i$ contains  at at least $m_{b'c'}-2 \geq 3$ 
syllables of $\widehat{(u_i)^\#}$, and we find $u,v$ just as above.
If $\beta_{i-1} \neq \emptyword$, then $w_i$ contains at least $m_{b'c'}-3$
syllables of $\widehat{(u_i)^\#}$ so, provided that $m_{b'c'}>5$, the argument
above still works to find $u,v$ within $w_i$. 
Similarly, if $m_{b'c'}=5$ and $b=b'$, $c=c'$, then $\beta_{i-1}$ can 
contain a power of $c$ but no power of $b$, and the next syllable of $\widehat{(u_i)^\#}$ to the right of that power of $c$ is within $w_i$ and is a power of $b$ ;
then we can choose $u$ and $v$ within $w_i$, as above.
In all these cases,  
$c^{\ee} b^{\kk} \beta_{i-1}$ is within $(u_i)_p(u_i)_q$ (see Definition~\ref{def:P3G}),and so $(u_i)_r$ is a suffix of $w_i$.

It remains to consider the case where $u_{i-1}$ is P3G-critical, and $m_{bc}=5$ with $b=c'$ and $c=b'$; then $u_{i-1}$ has type $(a,b,c)=(a,c',b')$.
For $\widehat{(u_i)^\#}$ to
have syllable length at least $5$, the word $w_i$ must still have a suffix
containing letters with names $c'$ and $b'$, 
with (by Definitions~\ref{def:P3G},~\ref{def:P3Gabr}) a letter with name $b'$ within $(u_i)_r$,and one with name $c'$  either to
the left of $(u_i)_r$ or within $\alpha((u_i)_r)$, and hence to the left of a
letter with name $b'$.

It remains to show that $\vv_i$ has a P2G-critical
suffix of type $\{a',b'\}$. If $\beta_{i-1}$ is a power of $c$, then once again
the P2G-critical word $(u_i)_r$ is a suffix of $w_i$ and hence of $\vv_i$.
Otherwise, by part (iii), we have $\beta_{i-1} = c^\ll (a')^\mm$,
with $m_{aa'} = 2$.  Now $u_{i-1}$ has a suffix $v$, whose first and last
letters have names $a'$ and $a$, from which, once we have 
deleted all letters in $\rho_{i-1}$, we obtain a word of the form 
$(a')^\mm a^\nn$ with $\nn \ne 0$. Since (by assumption) $u_{i-1}$ is
P3G-critical of type $(a,c',b')$, we know from the first statement of (iv)
(applied to $u_{i-1}$) that $w_{i-1}$ has a suffix 
containing both $b'$ and $c'$, and so (since $v$ contains neither $b'$
or $c'$) $v$ must be a proper suffix of that, and so a suffix of $w_{i-1}$.
Then, since $m_{aa'}=2$ we see that the word $vw_i$ is a P2G-critical suffix
of $\vv_i$ of type $\{a',b'\}$ with the letters with name
$a$ lying within $\alpha(vw_i)$.
\end{proof}

We note a few elementary consequences of Lemma~\ref{lem:RRSdetails}
\begin{corollary} \label{cor:RRSdetails2}
\begin{mylist}
\item[(i)] For $1<i<m+1$, $\l{u_{i-1}}$ cannot commute with both generators
in $P_i$.
\item[(ii)] For $1\leq i < m+1$, if $\alpha_i$ is non-empty then it 
	lies within $w_i$.
\end{mylist}
\end{corollary}
\begin{proof}
(i) If $u_{i-1}$ is P2G-critical then $\l{u_{i-1}} \in P_{i-1}$, and
$|P_i \cap P_{i-1}|>0$ by Lemma~\ref{lem:RRSdetails}\,(ii).
If $\l{u_{i-1}}$ is in $P_i$,
then it cannot commute with the other pseudo-generator of $P_i$.
Otherwise $\l{u_{i-1}}$ is not in $P_i$, and it does not commute
with the other pseudo-generator in $P_{i-1}$, which is in $P_i$.

If $u_{i-1}$ is P3G-critical then $P_i=P_{i-1}$, and by
Definition~\ref{def:P3G} $\l{u_{i-1}}$ commutes with one element
of $P_{i-1}$ but not with the other.

(ii) When $i=1$, the result holds trivially, since $w_1=u_1$.
	When $i>1$, by Lemma~\ref{lem:RRSdetails}(iii), $u_{i-1}$ cannot be P3G-critical
when $\alpha_i \neq \emptyword$, so $u_{i-1}$ is P2G-critical.
Then by Lemma~\ref{lem:RRSdetails}(ii), $\beta_{i-1}=\emptyword$ and so
$u_i=\l{\tau(u_{i-1})}w_i$.
\end{proof}

\begin{lemma}\label{lem:RRScheck}
We can check in linear time whether a given factorisation
$w=\mu w_1 \cdots w_mw_{m+1}x$ of a word $w$ is associated with an RRS $U$
of $w$.
\end{lemma}
\begin{proof} By Lemma~\ref{lem:abcrit2}, we can check in linear time whether
a given word could occur as the suffix $(u_i)_r$ of a critical word $u_i$ of
type $(a,b,c)$. Given this and Lemma~\ref{lem:critlintest},
we see easily that all of the types of $\tau$-moves can be
executed in linear time, and the result follows.
\end{proof}
 
\begin{example}~\label{eg:tricky_w_in_G_2}
We consider the word
$w=(b,c)_{n-1} \cdot aba \cdot  {}_{n-2}(c,b)x^{-1}$ that was considered 
in Example~\ref{eg:tricky_w_in_G}.
That admits an RRS of length 2, with associated decomposition $w_1w_2w_3x^{-1}$,
where $w_1=(b,c)_{n-1}aba$, $w_2={}_{n-2}(c,b)$ and $w_3$ is empty.

The word $w_1$ is a P3G-critical word of type $(a,b,c)$,
as we observed in Example~\ref{eg:P3G}, and
$\tau(w_1) = (c,b)_{n-1}acb$, as we computed in that example.

Now $u_2 = cb\cdot w_2 = {}_n(c,b)$, and $\tau(u_2)={}_n(b,c)$.
So application of the RRS to $w$ yields
\begin{eqnarray*}
w &\to& \tau((b,c)_{n-1} aba) \cdot {}_{n-2}(c,b)x^{-1}
=(c,b)_{n-1}acb\cdot  {}_{n-2}(c,b)x^{-1}\\
	&\to& (c,b)_{n-1} a \cdot  \tau({}_{n}(c,b))x^{-1} 
 = (c,b)_{n-1} a  \cdot {}_{n}(b,c)x^{-1},
\end{eqnarray*}
which ends in the cancelling pair $xx^{-1}$.
\end{example}

The following example with $n=4$ shows that we need $n$ to be at least $5$.

\begin{example} \label{eg:n_atleast5}
When $n=4$, the word $(bab)(cbca)(bc)b^{-1}$ admits the RRS
\[ bab, acbca,bcbc,bb^{-1}\] associated with the decomposition of length 3 that
corresponds to the bracketing shown, and hence allows the sequence of reductions
\begin{eqnarray*}
(bab)cbcabcb^{-1} 
	&\to& \tau(bab)cbcabcb^{-1} = ab(acbca)bcb^{-1}\\
	&\to& ab(c\tau(aba)c)bcb^{-1}=abcba(bcbc)b^{-1}\\
&\to& abcba\tau(bcbc)b^{-1}=abcbacbcbb^{-1}\to abcbacbc.\end{eqnarray*}
But $bacbcbabcb^{-1}$, which represents the same element of $G$, admits no RRS;
its only P2G- or P3G-critical subwords are $cbcb$ and $bab$, and neither of
those subwords work as $u_1=w_1$ in an RRS.
\end{example}

We define an \emph{optimal RRS} for a word in a $A_3,B_3$-free group $G$ as
follows 
(following \cite{BCMW}).
\begin{definition}[Optimal RRS]
\label{def:optRRS}
We say that an RRS $U=u_1,\ldots,u_{m+1}$ for $w$, with associated
decomposition $w=\mu w_1\cdots w_{m+1}\gamma$, is \emph{optimal} if
\begin{mylist}
\item[{\rm (i)}] the first letter of $w_1$ is at least as far to the right as in
any other RRS for $w$, i.e. the prefix $\mu$ is as long as
possible;
\item[{\rm (ii)}] $\l{\tau({u}_i)} \neq \f{w_{i+1}}^{-1}$ for $i \leq m$,
	and $\f{\gamma}$ does not appear in $w_{m+1}$;
\item[{\rm (iii)}] if, for some $i>1$ with $1 < i \le m$, $u_{i-1}$ is
	a P2G word of type $\{x,y\}$ and $\alpha_i$ commutes with both $x$ and $y$, then $u_{i}$ is 
\begin{mylist}
\item[either] a P2G word  of type $\{z,t\}$
with $|\{x,y\} \cap \{ z,t\} \} | = 1$ 
\item[or] an
$(a,b,c)$-critical word with $|\{x,y\} \cap \{b,c\} \} | = 1$.
\end{mylist}
\end{mylist}
\end{definition}
Note that our condition (iii) is slightly different  from the third condition
for optimality of \cite{BCMW}. 
We prefer to state our condition (iii) in this form for greater readability.
Note also that it is possible to have
two successive $\tau$-moves of type $(a,b,c)$ in an optimal RRS. 

We shall use the following properties of an optimal RRS.
\begin{lemma} \label{lem:properties_optRRS}
Let U be an optimal RRS of length $m$ and $1 \le i \le m$.
\begin{mylist}
\item[(i)] If $i>1$, the word $u_i$ has no proper critical suffix of the same type as
$u_i$ that is contained within $w_i$.
\item[(ii)] Let $x$ be the first letter in $u_i$ after $\f{u_i}$ that is
not part of $\alpha_i$. Then $x$ does not have the same name as $\f{u_i}$.
	If $i>1$ then $\alpha_i$ is a prefix of $\s{u_i}$.
\end{mylist}
\end{lemma}
\begin{proof} (i) If $u_i$ had such a suffix $w_{i2}$ with $w_i=w_{i1}w_{i2}$,
then we could define a new RRS $w_{i2},u_{i+1},\ldots,u_{m+1}$ for $w$ with
decomposition $\mu'w_{i2}w_{i+1} \cdots w_{m+1}\gamma$ where
$\mu' = \mu w_1\cdots w_{i-1}w_{i1}$, contradicting the optimality of $U$.

(ii) Let $b$ be the name of $f[u_i]$. If $i>1$ and either $u_{i-1}$ has type $(a,b,c)$ or $\beta_{i-1} \ne
\emptyword$, then $\alpha_i=\emptyword$ by Lemma~\ref{lem:RRSdetails}\,(ii) and
(iii), so $x$ is the second letter of $u_i$ and is part of
$\tau(u_{i-1})$, and we see directly from the definition of $\tau(u_{i-1})$
that $x$ is either $f[\beta_{i-1}]$ or has name $c$, so
has a different name from $\f{u_i}$. Otherwise $x$ is in $w_i$.
We cannot have $x = \f{u_i}^{-1}$ because in that case $\widehat{u_i}$ would not
be freely reduced and so would not be critical. If $x = \f{u_i}$ then $u_i$
would have a critical suffix of the same type as $u_i$ starting at $x$,
contradicting (i). So $x$ does not have the same name as $\f{u_i}$.

	Finally, by Definition~\ref{def:P2Gabr} any letter with name distinct
	from $\f{u_i}$ must come after every letter in $\alpha_i$, and hence it follows from the above, that $\alpha_i$ is a prefix of $\s{u_i}$.
\end{proof}

\begin{proposition} \label{prop:exists_optRRS}
If a word $w$ in a $A_3,B_3$-free group admits an RRS then it admits an
optimal RRS.
\end{proposition}
\begin{proof}
Part of this proof is essentially the same as the existence part of the proof of
\cite[Lemma 3.12]{BCMW} for $3$-free groups.
We can clearly restrict attention to those RRS for $w$ in which the first letter of
	$w_1$ is as far right as possible, and hence (i) holds. 

Then, given such an RRS associated to a decomposition
$\mu w_1\cdots w_{m+1}\gamma$ of $w$ for which $\l{\tau(u_i)}=\f{w_{i+1}}^{-1}$
for some $i \leq m$, we can associate the shorter RRS
$u_1,\cdots,u_i,\emptyword$
with the decomposition $\mu w_1\cdots w_i\gamma'$ of $w$, where
$\gamma'=w_{i+1}\cdots w_{m+1}\gamma$.  And given an RRS associated with a
decomposition of $w$ for which $\f{\gamma}$ appears in $w_{m+1}$ we can replace
$w_{m+1}$ by a prefix that ends before $\f{\gamma}$. 

So we can find an RRS for $w$ satisfying both conditions (i) and (ii).

Suppose we have an RRS satisfying conditions (i) and (ii), but not (iii),
for some $i$ with $1 \leq i \leq m$, and choose such an RRS such that its
length $m$ is as small as possible.
Then the hypothesis of this condition could apply only
when $\alpha_i$ commutes with the pseudo-generators of $u_{i-1}$ (which is
P2G-critical) and, for the conclusion to fail, either $u_i$ would
have to be a P2G word of the same type as $u_{i-1}$, or else $u_{i-1}$ and $u_i$
would be of types $\{b,c\}$ and $(a,b,c)$, respectively. In either of these
	cases we would have $P_i=P_{i-1}$ and $\beta_{i-1} = \emptyword$ by
Lemma~\ref{lem:RRSdetails}\,(ii). We shall prove that under these conditions
we can find a shorter RRS for $w$ that starts in the same place as the
original RRS, and the existence of an optimal RRS will
then follow from the minimality of $m$.

First we claim that the word $v:=u_{i-1}w_i$ is pseudo-generated with the same
pseudo-generators as $u_i$, and $\beta(v)=\beta_i$.
If $u_i$ is P2G, then, since $\beta_{i-i}=\emptyword$ and $\alpha_i$
commutes with both pseudo-generators of $u_{i-1}$, all internal letters of $v$
can be moved using commutation relations to the beginning 
and end of the word $v$,
so $v$ is P2G and $\beta(v)=\beta_i$.
If $u_i$ is P3G, then $v$ is P3G with $v_p=(u_{i-1})_p$,
$v_q=(u_{i-1})_q(u_{i-1})_s(w_i)_s$, where $(w_i)_s = \s{(u_i)_p(u_i)_s}$,
and $v_r=(u_i)_r$. Again we have $\beta(v)=\beta_i$.

Now, suppose first that $v$ has a P2G prefix $v_0$ such that $\widehat{v_0}$ is not
geodesic. Then $v$ has a prefix $v'$ such that the final letter
of $\tau(v')$ cancels the first pseudogenerator of $v$ after $v'$.
In that case, $w$ has a shorter RRS $u_1,\dots,u_{i-2},v',v''$, where $v''$ consists
of the internal letters of $v$ that lie between $\l{v'}$ and the 
pseudogenerator of $v'$ that cancels.
	So we may assume that $v$ has no such prefix.

Suppose next that $u_i$ (and hence also $v$) is P3G critical. Then,
since $v_r=(u_i)_r$, we know that $\widehat{v_r}$ can be
transformed by a $\tau$-move to a word of the form $b^Ia^Jb^K$. 
We recall that $v^\# = v_pv_q \alpha(v_r)\rho(v_r)b^I$, and we denote by
$\widehat{v^\#}$ the word formed from $v^\#$ by deleting all its letters
with
names other than $b$ and $c$. Recall also that $\alpha(v_r)$ could have letters
with name $c$, but all other letters in $\alpha(v_r)\rho(v_r)$ lie
in $\rho(v^\#)$ and commute with both $b$ and $c$. In particular,
$\beta(v^\#)$ is empty. From the preceding paragraph we know that
$\widehat{v_pv_q}$ is geodesic, but there are a number of possible situations
in which $\widehat{v^\#}$ is not geodesic:
\begin{mylist}
\item[(i)] $v_pv_q$ ends in a letter $c^{\pm 1}$ that cancels a letter
with name $c$ in $\alpha(v_r)$;
\item[(ii)] $v_pv_q$ ends in a letter $b^{\pm 1}$ that cancels with
the first letter of $b^I$;
\item[(iii)] $v_pv_q$ is P2G critical and $\tau(v_pv_q)$ ends in a letter
$c^{\pm 1}$ that cancels a letter with name $c$ in $\alpha(v_r)$;
\item[(iv)] $\widehat{v_pv_q\alpha(v_r)}$ is critical, and 
$\tau(\widehat{v_pv_q\alpha(v_r)})$ ends in a letter $b^\epsilon$ that cancels
with the first letter of $b^I$.
\end{mylist}
In fact in Case (i) the word $v$ has a prefix $v_0$, which is a prefix of
$v_{i-1}v_pv_q\alpha(v_r)$, such that $\widehat{v_0}$ is not geodesic,
and we have already dealt with that situation. In Cases (ii)--(iv) the
word $w$ has a shorter RRS, and the result follows by our choice of $m$.
For example, in Case (iv), we have the shorter RRS $v_1,v_2,\ldots,v_{i+1}$ of
length $i$, defined as follows. Let $(v_r)_p =v_av_b$, where $v_a$ is
its longest prefix of $(v_r)_p$ that ends with a letter with name $c$
(so $v_a$ is empty if there is no such letter). Then we set
$v_j:=u_j$ for $1 \le j < i-1$, $v_{i-1} := v_pv_qv_a$, the word $v_i$ a
P2G critical proper prefix of
$\l{\tau(\widehat{v_{i-1)}}}\beta(v_{i-1})v_b(v_r)_q(v_r)_s$,
and $v_{i+1}$ a subword of $v_r$ that commutes with $a$.
Note that $\beta(v_{i-1})$ is a (possibly zero) power $a$ coming from
$\alpha(v_r)$.  We leave Cases (ii) and (iii) to the reader.
So we may assume that $\widehat{v^\#}$ is geodesic when $v$ is P3G critical.

We complete the proof by showing that,
if we are not in any of the situations that we already considered,
then $v$ is critical of the same type as $u_i$ such that
$\beta(v)=\beta_{i}$ and, if $u_i$ is P2G, then
$l[\tau(\widehat{u_i})]=l[\tau(\hat{v})]$ while, if $u_i$ is P3G, then
$l[\tau(u_i^\#)]=l[\tau(v^\#)]$. So we can
combine $w_{i-1}$ and $w_i$ in the decomposition to obtain a shorter RRS $U'
=u_1',\ldots,u_{m-1}',u_{m}'$ of length $m-1$ with $u_j'=u_j$ for $j<i-1$,
$u_{i-1}' = u_{i-1}w_i$, and $u_j' = u_{j+1}$ for $i \le j \le m$.

Suppose first that $u_i$ is P2G critical.
Then, since $\hat{v}$ is geodesic and, by Lemma~\ref{lem:properties_optRRS},
$w_i$ does not contain any critical suffixes, we have
$$p(\hat{v})+n(\hat{v})=p(\widehat{u_i})+n(\widehat{u_i})=
p(\widehat{w_i})+n(\widehat{w_i})+1.$$ 
So either $p(\widehat{w_i})=p(\widehat{u_i})$ or $n(\widehat{w_i})=
n(\widehat{u_i})$. Assume without loss of generality that we are in the first case. If $\widehat{u_i}$ is positive then $\hat{v}$ also has to be positive and produce the same final letter. If $\widehat{u_i}$ is not positive, then it starts with a negative alternating sequence of length
$n(\widehat{u_i})=n(\widehat{w_i})+1$ and finishes with a positive alternating
sequence of length $p(\widehat{w_i})=p(\widehat{u_i})$.
This means that also $p(\hat{v})=p(\widehat{w_i})$, and $\hat{v}$
has to start with a negative alternating sequence of length $n(\widehat{u_i})$, so $v$ is P2G critical with
$\l{\tau(\hat{v})}=\l{\tau(\widehat{u_i})}$ as claimed.

Finally, suppose that $u_i$ is P3G critical.  Then, as we saw above, we may
assume that that the word $\widehat{v^\#}$ is geodesic.
Using a similar argument as in the preceding paragraph, we can show that
$p(v^\#) = p(u_i^\#)$ and $n(v^\#) = n(u_i^\#)$, so $v^\#$
is P2G critical as claimed, and $l[\tau(u_i^\#)]=l[\tau(v^\#)]$.
\end{proof}

\begin{definition}[The set $W$]
	\label{def:W}
We define $W$ to be the set of words that admit no RRS. 
\end{definition}
As we already observed, a word that is not
freely reduced admits an RRS of length $0$ with $w_{m+1} = \emptyword$,
so all words in $W$ are freely reduced. 
Since a word admitting an RRS is transformed to a freely reducing word by its 
application, any geodesic word must be in $W$. 

We describe as Procedure~\ref{proc:unique_optRRS} (below) a linear time procedure that, given a word $w \in W$ and 
$x \in A$, attempts to construct an optimal RRS that transforms $wx$ to a word
ending in $x^{-1}x$.
The procedure either returns a sequence of words $u_i$ and associated factorisation of
$wx$ defining an optimal RRS, or it returns \fail, in which
case $wx$ is proved to be in $W$.
We prove correctness of this procedure in Proposition~\ref{prop:unique_optRRS},
which follows the procedure.

The procedure has two principal parts. In the first part, either the message \fail\ is
returned, or the data associated with a putative (unique optimal)
RRS $U$ for $wx$ is computed.
In the second part, if \fail\ was not returned, then it is checked whether
the putative RRS $U$ really is an RRS for $wx$. If so, then the data
associated with $U$ are returned, but otherwise \fail.

The first part contains $m+1$ steps and constructs the
putative decomposition of $w$ from right to left; 
note that the value of $m$ is identified only when the $(m+1)$-st step is recognised as the final one of the first part.
This part attempts to identify a sequence 
of subwords $w_i$ of $w$, associated with critical words $u_i$,
for $i=m+1,m,\ldots,1$, in that order. 
Given the location of the right-hand end of $w_i$ and the
criticality type of $u_i$, the ($m+2-i$)-th step checks if $i=1$. 
If so then it attempts to locate the left-hand end of $u_1$, but otherwise it
attempts to locate the right-hand end of $w_{i-1}$ and determine the criticality type of $u_{i-1}$.

Before describing the procedure, we set up notation that will be used in the 
description of Step $m+2-i$ of part 1, when $i\leq m$ (that is, after Step 1). At the start of that step, the procedure has 
identified the right-hand end of the subword $w_i$ of $w$ (in a putative 
decomposition), the criticality type of $u_i$ and associated set $P_i$ of primary pseudo-generators for $u_i$,
and the criticality type of $u_{i-1}$. 

\begin{definition}[Notation for Step $m+2-i$]\label{def:step_notation}
We denote by $\vv_i$ the prefix of $w$ that ends at the end of $w_i$,
and by $s,t$ the elements of the set $P_i$
	(to be separately identified within the procedure). 
	If $u_i$ is a P3G-critical word, we suppose that it has type $(a',b',c')$,
	and hence that $\{s,t\}=\{b',c'\}$

During the step, a letter must be located at a
particular position within the prefix $\vv_i$ of $w$; we call that letter the
\emph{distinguished letter}, and denote it by $\disting$;
the precise definition of $\disting$ depends on the case and 
subcase within the step,
The letter $\disting$ always has the property that its name does not lie in $P_i$
and it commutes with at most one element of $P_i$. If it commutes with only
one such generator, then that generator is called $s$.

Particular letters are examined within the step that are found to the left and right of
$\disting$, which are called its \emph{left} and 
\emph{right quasi-neighbours}, and denoted by $\leftnbr$ and $\rightnbr$. 
The letter $\rightnbr$ is defined to be the closest letter to $\disting$ on its
right that has name in $P_i$. The letter $\leftnbr$ is defined, in the case
where $\disting$ commutes with $s$, to be the closest letter to $\disting$ on
its left that either has name $s$ or does not commute with $s$;
otherwise $\leftnbr$ is undefined.
\end{definition}

\begin{procedure}\label{proc:unique_optRRS}
Search a word for an optimal RRS.
\begin{mylist}
\item[{\bf Input:}] $w \in W$, $x \in A$. 
\item[{\bf Output:}] Either the message \fail\ (when $wx \in W$), or the words $u_i$ and the
decomposition $\mu w_1 \cdots w_mw_{m+1}x$ associated with the unique
optimal RRS for $wx$, together with the criticality types of the words $u_i$.
\end{mylist}

\smallskip
{\bf Step 1:} 

\begin{myproclist}
\item Identify $w_{m+1}'$ as the longest suffix of $w$ that 
commutes with $x$ and does not contain $x^{-1}$.
\item If $w_{m+1}'$ either contains $x$ or is the whole of $w$,
	then {\bf return \fail}.
\item{ Check if $m=0$:} if the letter immediately to the left of the
suffix $w_{m+1}'$ within $w$ is $x^{-1}$, then deduce that $m=0$, define
$u_1:=w_1 := x^{-1}w_{m+1}'$, return the associated
decomposition $w = \mu w_1 x$, and \stop.
\item Deduce that $m>0$. Define $w_{m+1}:=w_{m+1}'$,
and identify the location of the right-hand end of $w_m$ as the position within
$w$ immediately to the left of $w_{m+1}$.
Let  $s,t$ be the names of $x$ and the last letter in $w_m$, respectively.
\item Identify the criticality type of $u_m$ as P2G, $P_m$ as $\{s,t\}$.
\end{myproclist}

	Now {\bf for each of $\mathbf{ i:=m,\ldots,1}$}, given the location of the right-hand end of
$w_i$ and the criticality type of $u_i$, found in the previous step, 
	{\bf perform Step $\mathbf{m+2-i}$} (described below).

If during this step, the procedure fails either to prove that $i=1$ or
to find the letter $\disting$ or $\leftnbr{}$ (or in some cases letters to the
left of $\leftnbr{}$), as required, before 
	the left-hand end of $w$ is reached, then {\bf return \fail}.

\item 
	{\bf Step $\mathbf{m+2-i}$:} 
If $u_i$ has criticality type P2G then select case 1, otherwise select case 2.

\smallskip
{\bf Case 1 of Step $\mathbf{m+2-i}$:} the word $u_i$ is P2G-critical,
and $P_i := \{s,t\}$.
\begin{myproclist}
\item
Move leftwards from the right-hand end of $\vv_i$.
and select the appropriate one of seven subcases (a)--(g).\\
If, before letters with both names in $P_i$ have been seen, a letter
not in $P_i$ is seen that does not commute with $\l{\vv_i}$ but does commute 
with the other pseudo-generator in $P_i$, then identify this letter as 
the distinguished letter $\disting$ and select {\bf Subcase (a):}
\smallskip
\begin{myproclist} \item
	Identify the generator $t$ as the name of $\l{\vv_i}$, and $s$ as the name of
the other pseudo-generator.
Deduce that $i>1$, that the location of the right-hand end of
$w_{i-1}$ is at the distinguished letter, and that $u_{i-1}$ is P2G-critical,
	with $P_{i-1}$ consisting of $t$ and the name of $\disting$. 
	Continue to the next step.
\end{myproclist}
\smallskip
\item Otherwise, after seeing letters with both names in $P_i$,
	define $\disting$ to be the first letter found moving leftwards within 
	$\vv_i$ from this point onwards,
that does not lie in $P_i$ and commutes with at most
one element of $P_i$.
If, before the letter $\disting$ is found,
a suffix of $\vv_i$ is found that is P2G-critical of type $\{s,t\}$, then
	select {\bf Subcase (b):}
	\smallskip
	\begin{myproclist}
	\item Deduce that $i=1$, identify $u_1=v_1$ as that
P2G-critical suffix, proceed to {\bf Checking}.
	\end{myproclist}
	\smallskip
\item If $\disting$ commutes with neither element of $P_i$, then select {\bf Subcase (c): }
	\smallskip
\begin{myproclist}
\item Deduce that $i>1$, and 
identify the location of the right-hand end of $w_{i-1}$ at the position
of $\disting$.  Deduce that $u_{i-1}$ is P2G-critical,
	and identify the elements of $P_{i-1}$
as the name of $\disting$ together with whichever of $s,t$ is not the name
of $\rightnbr$.
Continue to the next step.
\end{myproclist}
\smallskip
\item
Otherwise, identify $s$ as the generator in $P_i$ that commutes with $\disting$,
$t$ as the other generator in $P_i$.
	Define $a$ to be the name of $\disting$, $b:=t$ and $c:=s$.
If $\rightnbr$ has name $s=c$ then proceed to {\bf Subcase (f/g)} below.
\item
Otherwise ($\rightnbr$ has name $t=b$),
if $\leftnbr$ has name $s$, and the suffix $u$ of $\vv_i$ beginning at the
position of $\leftnbr$ is P2G-critical of type $\{s,t\}$, then select
{\bf Subcase (d):} 
\smallskip
\begin{myproclist}
\item Deduce that $i=1$, define $u_1=w_1:= u$, and proceed to {\bf Checking}.
\end{myproclist}
\smallskip
\item
If $\leftnbr$ has name $s=c$, and the suffix $u$ of $\vv_i$ beginning at the
position of $\leftnbr$ is not P2G-critical of type $\{s,t\}$, then
proceed to {\bf Subcase (f/g)} below.
\item
Otherwise ($\leftnbr$ does not commute with $s=c$),
let $d_L$ be the name of
$\leftnbr$.
If $d_L \ne b$, $m_{ab}\ne 3$, or $m_{bc} < 5$ then 
	select {\bf Subcase (e):}
	\smallskip
	\begin{myproclist}
	\item Deduce that $i>1$, that the letter $\leftnbr$ is the rightmost 
		letter of $u_{i-1}$, and that $u_{i-1}$
		is P2G-critical  $P_{i-1}=\{d_L,s\}$.
Continue to the next step.
	\end{myproclist}
\smallskip
\item
	Otherwise
continue to read leftwards in $\vv_i$ past the letter $\leftnbr$. 
Look for a letter with name $c$ and then continue
leftwards looking for a letter with name $b$. 
If, before these two letters are seen, a letter is seen with name $a$, 
	then select {\bf Subcase (g)} below, but otherwise
	select {\bf Subcase (e)} (above).
\item	
{\bf Subcase (f/g)}: (the purpose of this subcase is to describe and distinguish 
between the two Subcases (f) and (g)).
\smallskip
\begin{myproclist}
\item  If $m_{bc}<5$, or $m_{ab}>3$, or $\vv_{i-1}$ has no P2G-critical
suffix of type $\{a,b\}$, then
	select {\bf Subcase (f):}
	\smallskip
	  \begin{myproclist}
	  \item Deduce that $i>1$, that the letter $\disting$ is the rightmost 
          letter of $u_{i-1}$, and that $u_{i-1}$
          is P2G-critical of type $\{a,b\}$.  
          Continue to the next step.
	  \end{myproclist}
\item  Otherwise,
	deduce that the letter $\disting$ is the rightmost letter of $u_{i-1}$,
	let $u'_1$ be the shortest P2G-critical suffix of
$\vv_{i-1}$ of
type $\{a,b\}$, and define $u_2' := \l{\tau(\widehat{u_1'})} \beta(u_1') w_i$.
	\item If $u_2'$ is P2G-critical of type $\{b,c\}=\{s,t\}$,
		then identify $i$ as $2$,
         $u_1=w_1$ as the subword $u'_1$ of $w$,
		$u_{i-1}=u_1=u'_1$ as P2G of type $\{a,b\}$, $P_{i-1}=\{a,b\}$ (so we are again in Subcase (f)),
	and proceed to {\bf Checking}.
        \item  Otherwise select {\bf Subcase (g):}
		\smallskip
	  \begin{myproclist}
	  \item Identify the criticality type of $u_{i-1}$ as P3G of type 
		  $(a,b,c)$, $P_{i-1}=\{b,c\}$.
          Continue to the next step.
          \end{myproclist}
\end{myproclist} 
\end{myproclist} 
{\bf Case 2 of Step $\mathbf{m+2-i}$:} the word $u_i$ is $(a',b',c')$-critical,
and $P_i := \{s,t\} = \{b',c'\}$
\begin{myproclist}
\item
Move leftwards from the rightmost
letter of $\vv_i$, look first for a letter with name $b'$, then for a letter
with name $c'$, then subsequently for another letter with name $b'$,
	and select the appropriate one of seven subcases (a)--(g).\\
If, before finding the third letter of the three letters
with names $b'$, $c'$ and $b'$, a letter $a$ is found with $a \ne b'$,
	$m_{b'a}=2$ and $m_{c'a}=3$, then select {\bf Subcase (a)}: 
	\smallskip
	\begin{myproclist}
	\item
	 Deduce that $i>1$. Identify the
location of this letter $a$ as the right-hand end of $w_{i-1}$, and the type
of $u_{i-1}$ as $(a,b,c)$-critical, where $b=c'$ and $c=b'$.
Continue to the next step.
	\end{myproclist}
\item
Otherwise, after finding the third of the specified letters, which has name
$b'$, continue to move leftwards, and look for the distinguished letter,
$\disting$, defined to be the first letter encountered that has
name outside $P_i$ and commutes with at most one generator in  $P_i$.
\item
As in Case 1 if, before finding $\disting$, a suffix of $\vv_i$ is found that
	is P3G-critical of type $(a',b',c')$, then select {\bf Subcase (b):}
	\smallskip
	\begin{myproclist}
	\item Deduce that $i=1$, identify $u_1=v_1$ as that P3G-critical suffix, proceed to {\bf Checking}.
	\end{myproclist}
\item
Otherwise continue exactly as in Case 1 of this step, with five further
subcases, (c)--(g) (except that Subcase (d) tests whether the suffix $u$
of $\vv_i$ of P3G-critical of type $(a',b',c')$ rather than P2G-critical of
	type $\{s,t\}$)).
\end{myproclist}

\medskip
{\bf Checking:} the part of the procedure checks whether the putative RRS $U$ associated with the
decomposition that has been found
really is an RRS for $wx$. Note that the check is unnecessary when $m=0$.

\begin{myproclist}
\item Define $u_1:= w_1$.
\item	For $i:=2,\ldots,m$, do the following:
  \item \begin{myproclist}
       \item If $u_{i-1}$ is a P2G-critical word,  
	   then $u_i := \l{\tau(\widehat{u_{i-1}})} \beta(u_{i-1}) w_i$;
   \item If $u_{i-1}$ is an $(a,b,c)$-critical word, then
    $u_i: = c^{\ee} b^{\kk}  \beta(u_{i-1}) w_i$,
    using the notation of Definition~\ref{def:P3G_tau}.
    \item If $u_i$ is not critical of the specified type or $i=m$ and
$\l{\tau(\widehat{u_i})} \ne x^{-1}$, then return \fail.
  \end{myproclist}
  \item Return the data associated with $U$.
\end{myproclist}
	
\end{procedure}

\begin{proposition}\label{prop:unique_optRRS}
Let $w \in W$ and $x \in X$.
Then Procedure~\ref{proc:unique_optRRS} with input $w$ and $x$
runs in linear time.
If $wx \not \in W$, then the procedure computes and outputs the decomposition
$\mu w_1 \cdots w_mw_{m+1}x$ of $w$ associated with the unique optimal RRS
$U=u_1,\ldots,u_m,u_{m+1}$ for the reduction of $wx$.
If $wx \in W$, then the procedure returns \fail, either because it fails
to compute a putative RRS $U$, or because $U$ fails the checking process.
\end{proposition}

\begin{proof}
It follows essentially from Lemmas~\ref{lem:critlintest},~\ref{lem:abcrit2}
and~\ref{lem:RRScheck} that the procedure executes in linear time.
It is not clear from the description above of Procedure~\ref{proc:unique_optRRS}
that Subcase (f/g) does not violate the linear time of the procedure
in the situation when the search for a P2G-critical suffix of $p_{i-1}$ of
type $\{a,b\}$ has to be carried out in many different steps of the procedure,
but with some care this can be sorted; we have chosen not to go into the
details that would sort this. 

We consider that it is clear from its description that the checking part of the
procedure is accurate.
So if it returns a decomposition, then this is certainly correct and
hence $wx \not\in W$. In other words, if $wx \in W$ then the procedure
correctly returns \fail.

So we may assume for the remainder of this proof that $wx \not \in W$ and
so, by Proposition~\ref{prop:exists_optRRS}, $wx$ has a decomposition
$\mu w_1 \cdots w_mw_{m+1}x$ associated with an optimal RRS
$U=u_1,\ldots,u_{m+1}$.
We shall prove below that, given such $w$ and $x$,
Procedure~\ref{proc:unique_optRRS} must complete without failing, and return
the selected decomposition and optimal RRS.
Since the same optimal RRS decomposition for $wx$ will be returned by the
procedure whichever decomposition is actually selected within this proposition,
it will follow that $wx$ has a unique optimal RRS.

We consider the $m+1$ steps of the first part of the procedure in the order in
which they are applied; we note that the value of $m$ will be determined
during the procedure, but is not known at its start.

The purpose of Step 1 of the procedure is to locate the subword $w_{m+1}$
within $w$
and check whether $m=0$. If $m>0$, then Step 1
determines the type of the critical word $u_m$.

For $i=m,\ldots,1$, the $(m+2-i)$-th step, given the location of the right-hand
end of $w_i$ within $w$ and the type of $u_i$, should determine whether
$i=1$ and, if so, determine the location of the left-hand end of $w_i=w_1$.
If $i>1$, then it should determine the location of the right-hand end of
$w_{i-1}$ as well as the criticality type of the word $u_{i-1}$.

We recall the notations of distinguished letter $\disting$,
of its nearest differently named neighbours $\leftnbr$ and $\rightnbr$,
	and of the prefixes $\vv_i$ of $w$ from Definition~\ref{def:step_notation}, and the definitions of $\vv_i$, $\leftnbr$ and $\rightnbr$  (relative to
$\disting$) also from there. But we note that the definition of 
$\disting$ must be taken from the description of the procedure, since it 
varies according to the subcase within the step.

\smallskip
{\bf Correctness proof for Step 1.} 

The fact that, in any factorisation associated with an optimal RRS, 
$w_{m+1}$ is as defined by the procedure
follows from Definitions~\ref{def:RRS} and ~\ref{def:optRRS}\,(ii).

If the letter immediately to the left of the suffix $w_{m+1}$ within $w$ is
$x^{-1}$, then it follows from Definition~\ref{def:RRS} that $wx$ has an RRS
with $m=0$, which is certainly optimal (we cannot start further right).
So $m=0$ for our chosen sequence, and the procedure correctly verifies that.
Otherwise, by definition we do not have an RRS of length $0$, and the procedure 
correctly verifies that.

So now we suppose that $m>0$.
We know that the last letter of $w_m$ must be found immediately before the
first letter of $w_{m+1}$, as located by the procedure,
also that $\l{w_m}$ must have name distinct from $x$. 

By Lemma~\ref{lem:RRSdetails}\,(i) $u_m$ cannot be P3G-critical.
So it must be P2G-critical and (using Lemma~\ref{lem:P2G_taufacts})
its pseudo-generators must be the names of $\l{u_m}$ and
$\l{\tau(\widehat{u_m})}$.
But $\beta_m=\emptyword$ by Lemma~\ref{lem:RRSdetails}\,(i), so
$\l{\tau(\widehat{u_m})} = \l{\tau(u_m)} = \f{u_{m+1}} = x^{-1}$
(by Definition~\ref{def:RRS}), and so the elements $s,t$ of $P_m$
are correctly identified in Step 1 of the procedure.

\smallskip
{\bf Correctness proof for Step m+2-i.}

Next we prove the correctness of Step $m+2-i$, for each of $i=m,\ldots,1$. 
We assume (by induction) that the procedure enters the step having
correctly located the right-hand end of $w_i$ (and of $\vv_i$)
and the type of the critical word $u_i$.

There are two cases to consider for each such step, in which $u_i$ is
P2G-critical of type $\{s,t\}$, or P3G-critical of type $(a',b',c')$.
Each of the two cases splits into several subcases.
Throughout the following proof we make frequent use when $i>1$ of the
definition of the word $u_i$ as specified in Definition~\ref{def:RRS},
and of the possibilities for the word $\beta_{i-1}$ given in
Lemma~\ref{lem:RRSdetails}\,(ii) and (iii).

\smallskip
{\bf Case 1 of Step $\mathbf{m+2-i}$:} the word $u_i$ is P2G-critical of type
$\{s,t\}$.

Suppose first that the suffix $w_i$ of $u_i$ contains (a power of) only one of
the two generators in $P_i$. That generator must be
the name of $\l{\vv_i} = \l{u_i}$; we denote it by $t$. 
This situation is clearly possible only when $i>1$ ($u_i$ must be 2-generated) and,
by considering the restrictions imposed by Lemma~\ref{lem:RRSdetails}\,(ii) and
(iii) on the prefix of $u_i$ that does not lie in $w_i$, we see that this could
happen only when (1) $m_{st} = 3$;
(2) $u_{i-1}$ is P2G-critical with
$P_{i-1}$ consisting of $t$ and the name of $\l{u_{i-1}}$;
and
(3) $\f{\beta_{i-1}}$ (which commutes with $\l{u_{i-1}}$) has name $s$.

In that case $\l{w_{i-1}}=\l{u_{i-1}}$ does not commute with $t$
but, by Definition~\ref{def:P2G}, $w_i$ must be a suffix of $(u_i)_s$, and
contain no internal letter that does not commute with $\l{w_i}$.
So we correctly locate $\l{w_{i-1}}$ as the distinguished letter $\disting$, as that was defined 
at the top of Subcase (a). It then follows from (2) that the type of $u_{i-1}$ and the set
$P_{i-1}$ are also exactly as found by the procedure.

Otherwise $w_i$ must contain letters with the names of both elements of $P_i$,
and the procedure will not find $\disting$ as defined at the top of 
Subcase (a), and so will use its subsequent (re)definition
as the first letter with name not in $P_i$ that is
encountered moving leftward within $\vv_i$ after letters with both names in $P_i$ have been seen, and which
commutes with at most one of the elements of $P_i$.
Suppose from now until the end of Case 1 that $\disting$ is defined in this way.

By Lemma~\ref{lem:properties_optRRS}\,(i) we have $i=1$ if and only if
$\vv_i$ has a critical suffix that is P2G-critical of type $\{s,t\}$. We may
find such a suffix before we find the letter $\disting$, and then we are in 
Subcase (b), and this step of the procedure completes successfully.
As we shall see, when $i=1$ we may also find $\disting$ within $\alpha_1$.

From now until the end of Case 1, suppose that we are not in Subcase (b).

We claim that, provided that $i>1$, 
the procedure will find $\disting$ either
as a letter of $\alpha_i$ within $w_i$ or as $\l{w_{i-1}}$, and if $i=1$ (since we are not in Subcase (b)),
it must find $\disting$ within $\alpha_1$.
To see this,
suppose first that $\disting$ is found within the subword $w_i$. In that case,
$\disting$ can also be found within $u_i$ (which contains $w_i$ as a suffix, is equal to $w_i$ when $i=1$)
as an internal generator.
It must be to the left of the suffix $(u_i)_s$ of $u_i$, which contains only 
one generator of $P_i$, so cannot be in $\beta_i$, cannot be in $\rho_i$ (by 
definition), and so must be in $\alpha_i$.
Conversely, if moving leftwards through $u_i$ (rather than $\vv_i$),  a 
letter satisfying the conditions on $\disting$ were found within $\alpha_i$, then $\alpha_i$ would be
in $w_i$ by Corollary~\ref{cor:RRSdetails2}(ii), (and hence in $\vv_i$),  and so the procedure would
identify that letter as $\disting$; in that case $\disting$ would commute with $\f{u_i}$.
So the procedure will find $\disting$ within $w_i$ precisely when $\disting$ 
is within $\alpha_i$.
If $i=1$, then, since we are not in Subcase (b), the procedure must find $\disting$ within $w_1$, and now since $w_1=u_1$, we deduce as above that in this case
$\disting$ must be found within $\alpha_1$.

Now suppose that $i>1$.
If
$\l{w_{i-1}}\in P_i$, then $u_{i-1}$ must be P2G-critical of type $\{s,t\}$. 
For if $u_{i-1}$ is P3G-critical of type $(a,b,c)$ then by 
Definition~\ref{def:P3G} $\l{u_{i-1}}$ has name $a$,
but by Lemma~\ref{lem:RRSdetails} (iii) $P_i=P_{i-1}=\{b,c\}$, 
so $a \not \in P_i$.
Further, in this case, both $\l{u_{i-1}}=\l{w_{i-1}}$ and $\f{w_i}=\l{\tau(u_{i-1})}$ are within $P_i\cap P_{i-1}$, so $P_{i-1}=P_i=\{s,t\}$.
Then, by Definition~\ref{def:optRRS}\,(iii), $\alpha_i$ must contain at least
one letter
that does not commute with both $s$ and $t$, and the rightmost such letter will be found as $\disting$.
On the other hand, if  $\l{u_{i-1}} = \l{w_{i-1}}$ is not in $P_i$, then it
cannot commute with both $s$ and $t$ by Corollary~\ref{cor:RRSdetails2}\,(i).
So in this case, if $\disting$ is not found within $w_i$, then it
will be found as $\l{w_{i-1}}$.
This completes the proof of the claim.

From now until the end of Case 1, we may assume that we find the letter $\disting$.

Suppose first that $\disting$ commutes with neither $s$ nor $t$. Then 
$\disting$ is not within $w_i$, and must be 
$l[w_{i-1}]=l[u_{i-1}]$.
Then $u_{i-1}$ needs to be P2G-critical (or, by Lemma~\ref{lem:RRSdetails} (iii), we 
would have $P_{i-1}=P_i$ and by Definition~\ref{def:P3G} the final letter of $u_{i-1}$ must commute with 
one of the elements of $P_{i-1}$, and hence with one of $s,t$) and the name of $\disting$  is in $P_{i-1}$.
If $\beta_{i-1} \neq \emptyword$, then
we know by Lemma~\ref{lem:RRSdetails}\,(ii) that the name of $f[\beta_{i-1}]$ is $s$ or $t$, but by Definition~\ref{def:P2Gabr} $\beta_{i-1}$ must commute with 
$\l{u_{i-1}}=\disting$, and $\disting$ commutes with neither $s$ nor $t$; so we must have
$\beta_{i-1} = \emptyword$. Then by Definition~\ref{def:RRS},
$u_i = \l{\tau(\widehat{u_{i-1}})}  w_i$, and the second element of $P_{i-1}$
is $\f{u_i} \in P_i$.
Now by Lemma~\ref{lem:properties_optRRS}\,(ii), 
the first letter of $u_i$ to the right of $\f{u_i}$ that is not 
within $\alpha_i$ has name distinct from $\f{u_i}$. By 
Definition~\ref{def:P2Gabr}, there are no further letters of $\alpha_i$ to 
the right of this letter, and hence  it must be to the right of $\disting$, and be $\rightnbr$; it follows that the second element of $P_{i-1}$ must be
whichever of $s$ and $t$ is not the name of $\rightnbr$. So we are in
Subcase (c) of the procedure, and 
this step of the procedure completes successfully.

Otherwise $\disting$ commutes with $s$ but not with $t$ (by Definition~\ref{def:step_notation}).
We assume this for the remainder of our consideration of Case 1.
As in the procedure, we let $a$ be the name of $\disting$ and define
$b:=t$ and $c:=s$.
Before proceeding further,
we make some observations 
that we shall use repeatedly below, also valid until the remainder of Case (1).
Recall from above that either (A) $\disting$ is the rightmost letter of 
$w_{i-1}=u_{i-1}$ (in which case $i>1$) or (B) $\disting$ is within $\alpha_i$.
We make the following observations about these two cases.

In Case (A),
the generator $a$ is a
pseudo-generator of $u_{i-1}$. Further,
	if $u_{i-1}$ is P2G-critical then its
other pseudo-generator is in $P_i$ and does not commute with $a$, hence
must be $t=b$, and we are in Subcase (f). 
	while if $u_{i-1}$ is
P3G-critical then it must have type $(a,b,c)$, and we are in Subcase (g).

In Case (B), as we saw earlier,
the letter $\f{u_i}$ commutes with $\disting$ and so has name $s$. Furthermore, since $\alpha_i$ is
non-empty in this situation, it follows from
Lemma~\ref{lem:RRSdetails}\,(ii) and~(iii) that, if $i>1$, then $u_{i-1}$ is
P2G-critical with $\beta_{i-1} = \emptyword$, and then by Lemma~\ref{lem:properties_optRRS} $\alpha_i$ must be a prefix of $w_i$.

Now, suppose first that $\rightnbr$ has name $s=c$.
In that case, $\disting$ cannot be in $\alpha_i$ 
For if it were,
by (B) above we would have 
$\f{u_i}=s^{\pm 1}$, in which case
by Lemma~\ref{lem:properties_optRRS}\,(ii) the first letter after $\f{u_i}$ 
that is in $u_i$ but not in $\alpha_i$ could not have name $s$; but by Definition~\ref{def:step_notation}, $\rightnbr$
would be that letter.
It follows that $i>1$ and $\disting$ is $\l{w_{i-1}}$ in this case, and it follows from our observation in (A) above  we are in Subcase (f) or (g).

Otherwise $\rightnbr$ has name $t=b$. 

If $i=1$ 
then 
$\disting$ must be in $\alpha_i$, as in (B) above.
In that case the procedure finds $\leftnbr$ as the letter $\f{u_1}=\f{w_1}$, with
name $s$, and this step of the procedure completes successfully in Subcase (d).
So we may assume from now until the end of this case that $i>1$.

Now, at least one of
$s$ and $t$ must be a pseudo-generator of $u_{i-1}$ and so there must be a 
letter with name $s$ or $t$ to the left of $\disting$. Then, the procedure will
certainly find the letter $\leftnbr$, and we let $d_L$ be its name.

Suppose that $d_L=s$. If $\disting$ were in $\alpha_i$, 
then  $\f{u_i}=\l{\tau(\widehat{u_{i-1}})}$ would have name $s$ (by (B) above)
and so $\l{u_{i-1}}=\l{w_{i-1}}$ would have name distinct from $s$, so could not be $\leftnbr$ (as just claimed).
So $\disting$ must be $\l{w_{i-1}}$, and by (A) above, we are in Subcase (f) or (g) of the
procedure, which we consider further below.
Note that $\leftnbr$ is in $\beta_{i-1}$ in those subcases.

So now we suppose that $d_L$ does not commute with $s=c$. 
Suppose that $\disting$ 
is the rightmost letter of $w_{i-1}$, as in (A) above.
If $u_{i-1}$ is a P2G-critical word then (as we saw in our observations above) it has type
$\{a,b\}$, $\f{u_i}$ has name $b$,
and $\rightnbr$ must be the first letter in $w_i$ that has name $b$.
But now, 
by Lemma~\ref{lem:properties_optRRS}\,(ii), the first letter of $u_i$ after $\f{u_i}$ that is not in $\alpha_i$ cannot have name $b$, and so cannot  
be $\rightnbr$.
Hence
it follows from Definition~\ref{def:RRS} that 
$\beta_{i-1} \ne \emptyword$,
and since $d_L \ne c$, we must be in the second case of
Lemma~\ref{lem:RRSdetails}\,(ii) with $\beta_{i-1} = c^\ll d^\mm$ for some
non-zero $\ll$ and $\mm$. In that case we would have $d_L \ne b$,
$m_{ac}=m_{ad_L}=m_{bd_L}=2$, $m_{bc}=3$ and hence $m_{ab},m_{cd_L} \ge 5$
by the $A_3$,$B_3$-free hypothesis. But in fact that situation cannot arise at this
stage of the procedure, because we would be in Subcase (a), that is, we would have  encountered $\disting$ before encountering an instance of a letter with name $c\in P_i$, and this has
already been considered.

We conclude that $u_{i-1}$ cannot be P2G-critical.
So, by (A) above, $u_{i-1}$ is a P3G-critical word 
of type $(a,b,c)$ and we are in Subcase (g) with $d_L=b$, $m_{ab}=3$ and $m_{bc} \ge 5$. Notice that we still need to see that the procedure correctly gives this information, which we shall see further below. 

So now, still supposing that $d_L$ does not commute with $s=c$,
we suppose that
$\disting$ is in $\alpha_i$. Then, by (B),
the letter $\leftnbr$ with name $d_L$ is the rightmost letter of $w_{i-1}$,
and $u_{i-1}$ is P2G-critical of type $\{c,d_L\}$ with $\beta_{i-1} = \emptyword$.
We are in Subcase (e) in this situation.
(Note that this is the only situation in Case 1 in which $u_{i-1}$ and $u_i$
could possibly be P2G-critical of the same type $\{s,t\}$, but since the letter
$\disting$ in $\alpha_i$ does not commute with both $s$ and $t$, condition
(iii) of Definition~\ref{def:optRRS} (optimality of RRS) is not violated.) 

We are going to see now that the procedure correctly identifies the information given in the two previous paragraphs. If $d_L\neq b, m_{ab}\neq 3$ or $m_{bc}<5$ then we clearly are in the case (e). If $d_L=b$, $m_{ab}=3$ and
$m_{bc} \ge 5$, when we could be in Subcase (e) or (g). The procedure reads leftwards
in the word after $\leftnbr$, and looks first for a letter with name $a$ or $c$.
If a letter with name $a$ is seen first then, since $\beta_{i-1} = \emptyword$ in Subcase (e),
we must be in Subcase (g). But one with name $c$ is seen first then we may
still be in Subcase (g) with that letter $c$ in $\alpha((u_{i-1})_r)$. In
that case the procedure continues to read leftwards looking for a letter with name $a$
or $b$. In Subcase (g) the letter with name $a$ is seen, as
$\f{(u_{i-1})_r}$, before one with name $b$.  We claim that in Subcase (e)
one with name $b$ is seen first. To see this observe that, since $m_{bc} \ge 5$,
the only way that one with name $a$ could be seen first would be if that
letter named $a$ were $\l{u_{i-2}}$ with $i>2$ and $u_{i-2}$ P3G-critical
of type $(a,b,c)$. But in that case by Definition~\ref{def:RRS} we would
have $u_{i-1} =c^\ee b^\kk\beta_{i-2}w_{i-1}$ which, since $\beta_{i-2}$ would
be a power of $c$ and the suffix of $\widehat{u_{i-1}}$ in $w_{i-1}$ 
would be
a power of $c$ followed by a power of $b$, could not be $\{b,c\}$-critical.
So again the procedure correctly distinguishes between these two possibilities.

Finally, we need to prove that the procedure distinguishes between the cases (f) and (g) that we have not treated yet. 
Suppose that $m_{ab} >3$, or $m_{st}<5$, or $p_{i-1}$ has no P2G-critical suffix of
type $\{a,b\}$.
In Subcase (g), by definition we have $m_{ab}=3$ and $m_{bc}=m_{st}\geq 5$,
and by the final statement of Lemma~\ref{lem:RRSdetails}\,(iv)
applied to $u_{i-1}$, the word
$p_{i-1}$ has a suffix that is P2G-critical of type $\{a,b\}$.
Then $u_{i-1}$ must be P2G-critical, and
the procedure correctly verifies that we are in Subcase (f).

Otherwise, reading leftwards through the word $w$, starting at the position of
the rightmost letter of $p_{i-1}$, the procedure finds the shortest suffix
$u_1'$ of $p_{i-1}$ that is P2G-critical of type $\{a,b\}$.
Then either $u_{i-1}$ is P3G-critical of type $(a,b,c)$  or,
by Lemma~\ref{lem:properties_optRRS}\,(i), it is P2G-critical with $i-1 = 1$.
The procedure checks for the second of these two possibilities by applying
$\tau$ to $u_1'$, and checking whether this results in a critical word
$u_i = u_2'$ of the required type.
We claim that, if this word $u_2'$ is found to be critical, then we must be in
Subcase (f), as reported by the procedure.

Now the subword $w_i$ of $w$ has been determined, and in Subcases (f) and (g)
we would have $u_i = b^\epsilon \beta_{i-1} w_i$ and $u_i =
c^\ee b^\kk\beta_{i-1} w_i$, respectively, where $\beta_{i-1}=\beta(u_1')$
and $\epsilon = \pm 1$ has the same sign as $\kk$. Since, by
Lemma~\ref{lem:properties_optRRS}\,(i), the word $w_i$ cannot be critical
of the required type, it is easy to see that these two candidates for $u_i$
cannot both be critical of the required type.  So if the word $u_2'$ is
critical, then we cannot be in Subcase (g), and so we must be in Subcase (f).
This verifies the claim.

\smallskip
{\bf Case 2 of Step $\mathbf{m+2-i}$:} the word $u_i$ is P3G-critical of type
$(a',b',c')$.

We recall that, in this case, the procedure defines $\disting$ to be the first
letter that is encountered (moving leftwards within $\vv_i$) after letters with
names first $b'$, then $c'$, and then $b'$ have been seen, such that $\disting$
has name outside $P_i$ and commutes with at most one generator in $P_i$.

If $i=1$, then it is clear from Definition~\ref{def:P3G} of a P3G-critical word
that, after seeing letters with names $b'$ and $c'$, we cannot see a letter with
name $a \ne b'$ that does not commute with $c'$ before seeing another letter
with name $b'$, and so we cannot be in Subcase (a).

If $i>1$ then we apply Lemma~\ref{lem:RRSdetails}\,(iv) to see that, as we move leftwards in
$\vv_i$, we must see a letter with name $b'$ followed by a letter with name $c'$
within $w_i$,
and if we then  see $\l{w_{i-1}}$ before seeing another letter with name $b'$
then we must have $m_{b'c'}=5$ and $u_{i-1}$ P3G-critical of 
type $(a,c',b')$, where $a$ is the name of $\l{w_{i-1}} = \l{u_{i-1}}$
and $\beta_{i-1} = c^\ll (a')^\mm$.
Conversely, if, before seeing another letter with name $b'$, we see a letter
with name $a$ with $m_{ac'}=3$ and $m_{ab'}=2$, then this letter cannot be
an internal letter of $u_i$, so it must be $\l{w_{i-1}}$, and we are in
the situation described, which is Subcase (a) of the procedure.

Otherwise, after seeing letters with names $b'$ and $c'$, we see another letter
with name $b'$ within $w_i$. The letter with name $c'$ must either lie in
$\alpha((u_i)_r)$ or to the left of $(u_i)_r$ and, in either case this new
letter with name $b'$ cannot lie within $(u_i)_r$, so it must be in the prefix
$(u_i)_p(u_i)_q$ of $u_i$. This prefix is also a prefix of the P2G-critical
word $u_i^\#$. The correctness proofs for cases (b)--(f) of the procedure are
now the same as in Case 1, but applied to $u_i^\#$ rather than to $u_i$.

This completes the proof of Proposition~\ref{prop:unique_optRRS}.
\end{proof}

\section{Proof of the main theorem}
\label{sec:proofs}

We recall Definition~\ref{def:W} of $W$ as the set  of words admitting no RRS.
Our main theorem will be an immediate consequence of the following result,
which specifies the rewrite system $\cR$.
\begin{theorem}
\label{thm:main_details}
Let $G$ be an Artin group defined over its standard generating set $S$ for which the associated Coxeter diagram contains no $A_3$ or $B_3$ subdiagram,
and let $W$ be as defined above.
Then the words in $W$ are precisely the words that are geodesics in $G$.

There is a process that, given an input word $w:= x_1\cdots x_n$,
runs in quadratic time to find geodesic representatives $v_0,\cdots, v_n$ in
$W$ of successive prefixes of $w$, in each case by applying an RRS of
$\tau$-moves to the word $v_{i-1}x_i$, and hence finding a geodesic
representative $v_n$ for $w$, and thereby solving the word problem in $G$.

Furthermore, if $w,w' \in W$ represent the same element of $G$ then there is
a sequence of commutations and $\tau$-moves applied to critical 2-generator
subwords 
on non-commuting generators
that transforms $w$ to $w'$. 
\end{theorem}

This section contains the proof of Theorem~\ref{thm:main_details}.
We describe the steps of the proof now. The proof uses the results of Section~\ref{sec:RRS} together with details that are verified in
Propositions~\ref{lem:6.1}--\ref{lem:6.4} that follow. 

We need to prove that a word $w$ is in the set $W$ of words admitting no RRS
(see Definition~\ref{def:RRS}) if and only if it is geodesic.
We observed as we defined $W$ that it must contain
all geodesic words. We recall that if a word $w$ admits an RRS then it admits
an optimal RRS which, by Proposition~\ref{prop:unique_optRRS}, is unique for
words $w$ with $\p{w} \in W$.

In order to establish the first statement of Theorem~\ref{thm:main_details},
we need to show that a non-geodesic word must admit an RRS.  To do that, we 
define a binary relation $\sim$ on $A^*$ by $w \sim w'$ if and only if $w'$ can
be obtained from $w$ by carrying out a sequence of commutations and $\tau$-moves
on critical 2-generator subwords. It is clear that $\sim$ is an equivalence 
relation, and Propositions~\ref{lem:6.1}\,(i) and~\ref{lem:6.2}\,(i) below
ensure that it restricts to an equivalence relation on $W$. For $w \in W$,
we denote the equivalence class of $w$ by $[w]$ and we denote the set of all
equivalence classes within $W$ by $\Omega$.

We aim to prove that $\Omega$ corresponds bijectively to $G$, where $g \in G$
corresponds to $[w]$ for any geodesic representative $w$ of $g$. Our strategy
is to define a right action of $G$ on $\Omega$. This action will have the
property that, whenever $x \in A$ and $w,wx \in W$, we have $[w]x=[wx]$.
We shall now complete the proof of Theorem~\ref{thm:main_details}, and
provide the details of the action subsequently.

Let $w,w'\in W$ be two representatives of the same element of $G$.
Then, since all prefixes of $w$ and $w'$ lie in $W$, we see that
the elements $[w]$, $[w']$ of $\Omega$ are equal to the images of
$[\emptyword]$ under the action of the elements of $G$ represented by $w,w'$.
Since $w,w'$ represent the same element, it follows that $[w]=[w']$.
This establishes the final assertion of Theorem~\ref{thm:main_details}.
Conversely, the definition of $\sim$ ensures that any two words related by
$\sim$ must represent the same element.  Hence we have a correspondence
between the elements of $G$ and the equivalence classes of $\sim$ on $W$.

But now if $w,w' \in W$ represent the same element, with $w'$ geodesic 
but $w$ non-geodesic then, since the two words represent the same element we
must have $[w]=[w']$, so $w \sim w'$. But the definition of $\sim$ ensures
two equivalent words have the same length, which is a contradiction.
It follows that every word in $W$ is geodesic.

The proof of the first statement within Theorem~\ref{thm:main_details} is now
complete. So now suppose that $w=x_1\cdots x_n$ over $A$ is input.
Setting $w_0 = \emptyword$, for each $i=1,\ldots, n$ we apply
Proposition~\ref{prop:unique_optRRS} to find in time 
linear in $i$ a word $v_i \in W$ that represents $v_{i-1}x_i$. Hence in time
quadratic in $n$, we find $v_n \in W$ that represents $W$. The first statement
within Theorem~\ref{thm:main_details} ensures that $v_n$ is geodesic,
and hence we can decide whether $w=_G 1$. This completes the proof of the second statement of 
Theorem~\ref{thm:main_details}.

We turn now to the definition of the right action of $G$ on $\Omega$.
We start by specifying the action of an element $x \in A$ on the equivalence
class $[w]$ of a word $w \in W$. To show that this is well-defined,
it is sufficient to verify that $[w]x = [w']x$ for words $w,w' \in W$ for
which a single commutation or $\tau$-move on a critical 2-generator subword
transforms $w$ to $w'$.

If $wx \in W$ then, since $wx \sim w'x$, we have $w'x \in W$
and we define $[w]x := [wx] (= [w'x])$.  

If $wx \not \in W$ then $w'x \not\in W$. Then by
Proposition~\ref{prop:unique_optRRS} $wx$ has a unique optimal RRS, which
transforms $wx$ to a word of the form $vx^{-1}x$ where,
by Propositions~\ref{lem:6.1}\,(i) and~\ref{lem:6.2}\,(i),
$vx^{-1} \in W$ and hence $v \in  W$. Furthermore, $w'x$ has a unique optimal
RRS which, by Propositions~\ref{lem:6.1}\,(ii)\,(b)
and~\ref{lem:6.2}\,(ii)\,(b) is the RRS $V$ specified there (note that the
words $wx$ and $w'x$ here are the words referred to as $w$ and $w'$ in
Propositions~\ref{lem:6.1} and~\ref{lem:6.2}). Then, by
Propositions~\ref{lem:6.1}\,(ii)\,(c) and~\ref{lem:6.2}\,(ii)\,(c),
$V$ transforms $w'x$ to a word of the form $v'x^{-1}x$ with $v \sim v'$.
In this case we define $[w]x := [v] (= [w']x = [v'])$.

In either case, we have a well defined image of $[w]$ under $x$.

We need to consider the effect of applying first $x$ and then $x^{-1}$ to $[w]$.
We consider separately the two cases $wx \in W$ and $wx \not \in W$.

If $wx \in W$, then $wxx^{-1} \not \in W$, and indeed
$wxx^{-1}$ admits an RRS of length $0$ that transforms it to $w$.
Hence by definition $[wx]x^{-1} = [w]$, and so in this case
$([w]x)x^{-1} = [w]$.
If $wx \not \in W$ then, as we saw above, we have $w \sim vx^{-1}$ for
a word $v \in W$ of length $|w|-1$, and we defined $[w]x = v$.
So again we have $([w]x)x^{-1} = [vx^{-1}] = [w]$.
Hence the maps $[w]\mapsto^x [w]x$ extend to an action of the free group of rank
$|S|$ on $W$, for which $[w]u = ([w]u')x$ for any word $u=u'x$ over $A$.

Our next step is to verify that, for each $w \in W$, we have
$[w]ef=[w]fe$ for commuting generators $e$ and $f$, and
$[w]\cdot {}_n(s,t)=[w]\cdot {}_n(t,s)$ for non-commuting generators
$s$ and $t$ with ${}_n(s,t)={}_n(t,s)$.
These results are proved as Propositions~\ref{lem:6.3},~\ref{lem:6.4} below;
those propositions imply that our action of the free group on $\Omega$ induces
to an action of $G$ as required. This finishes 
the proof of the theorem.

We provide proofs of the Propositions~\ref{lem:6.1}--\ref{lem:6.4} below. These
correspond to Lemmas (6.1)--(6.4) of \cite{BCMW}.
Our arguments were guided by those of \cite{BCMW} and we have kept
our notation as close to the notation of that article as possible, in order to
aid comparison.

\begin{proposition}\label{lem:6.1}
Suppose that the word $w \in A^*$ admits an optimal RRS
$U=u_1,\ldots,u_m,u_{m+1}$
with decomposition $w=\mu w_1 \cdots w_mw_{m+1}\gamma$.
Suppose also that $w$ has a subword $\zeta$ that is the product of two
commuting elements of $A$, and let $w'$
be the word obtained from $w$ by replacing $\zeta$ by its reverse.
Then:
\begin{mylist}
\item[(i)] The word $w'$ admits an RRS of length $m$;
we denote that by $V$, and the associated decomposition by
$w' = \nu y_1 \cdots y_{m+1}\eta$.
\item[(ii)] Suppose that $\p{w} \in W$, so that $\gamma$ is a single letter
$\xx$, and that $\zeta$ is within $\p{w}$.
Then the RRS $V$ can be chosen to ensure that:
  \begin{mylist}
  \item[(a)] $\eta=\gamma=\xx$;
	  \quad{(b)}\quad $V$ is optimal;
  \item[(c)] if $w \rightarrow_U vx^{-1}x$ and $w' \rightarrow_V v'x^{-1}x$, 
               then $v$ can be transformed to
               $v'$ using only commutation relations of $G$.
   \end{mylist}
\end{mylist}
\end{proposition}

\begin{proof} 
We observe first that, if the hypotheses of part (ii) of the proposition hold,
then it follows from the fact that $\zeta$ is within $\p{w}$
that $\l{w}=\l{w'}$.
Further, in this case, we cannot have $\p{w'} \not \in W$, since then it 
	would follow from part (i) of the proposition applied to $\p{w}$ 
	that $\p{w} \not \in W$; so $\p{w'} \in W$ and $\eta$ is a single 
	letter, and hence (ii)(a) follows immediately from (i).

Thus we need only prove parts (i), (ii)\,(b) and (ii)\,(c).

We write $\zeta^\Rev$ for the reverse of $\zeta$.
We derive the RRS $V$ of part (i) and associated decomposition for $w'$ by
examining the effect that replacement of the subword $\zeta$ of $w$ by
$\zeta^\Rev$ has on the factorisation $\mu w_1\cdots w_{m+1}\gamma$.

We split the proof into a list of cases that depend on
where the subword $\zeta$ occurs in the decomposition of $w$.
Note that, by the start  or beginning of a subword of $w$ or $w'$, we mean
the location of its leftmost letter and, by the end, we mean the location
of its rightmost letter.

\smallskip
{\bf Case 1:} $\zeta$ is a subword of $\mu$ or of $\s{\gamma}$.

\begin{addmargin}[1em]{0em}
$w'$ has $U$ as an optimal RRS associated with a decomposition of $w'$ of the
form $\nu w_1\cdots w_{m+1}\eta$, and the results are straightforward.
\end{addmargin}

\smallskip
{\bf Case 2:} $m>0$, and $\zeta$ overlaps the right-hand end of $\mu$ and the
left-hand end of $w_1$.

\begin{addmargin}[1em]{0em}
Recall that (by definition) $w_1=u_1$.
We define $y_1:=\f{u_1}\f{\zeta}\s{u_1} $, $v_1 := y_1$,
$y_i:= w_i, v_i := u_i$ for $i>1$, $\nu := \p{\mu}$, $\eta := \gamma$.
So $\nu y_1\cdots y_{m+1}\eta = w'$.
It follows from Definitions~\ref{def:P2G} and ~\ref{def:P3G}
that $v_1$ is critical of the same type as $u_1$, with $\f{\zeta}=\l{\mu}$
forming part of $\alpha(v_1)$,
and by definition we have $\tau(y_1)= \f{\zeta}\tau(u_1)=\l{\mu}\tau(u_1)$.  
So $V$ is an RRS.

To prove that $V$ is an optimal RRS for $w'$ given the hypotheses for
part (ii), it remains to verify those conditions of
Definition~\ref{def:optRRS} for optimality that refer to $y_1=v_1$;
we just need to check conditions (iii) and (i).
Since $v_1$ has the same criticality type as $u_1$ and
Definition~\ref{def:optRRS}\,(iii) holds for $U$, it must also hold for $V$.

	Suppose that  Condition (i) of Definition~\ref{def:optRRS}
fails for $V$. In that case, Procedure~\ref{proc:unique_optRRS} ensures that 
the unique optimal RRS for $w'$ is associated with a decomposition
$\nu'y'_1y_2\cdots y_{m+1}\eta$ of $w'$, with $y'_1$ a proper suffix of $y_1$.
We see from Procedure~\ref{proc:unique_optRRS} that the criticality type
of $u_1$ is determined either by $(u_1)_r$ or, in the case when this type
could be $\{a,b\}$ or $(a,b,c)$, by whether $\vv_1=\nu_1y_1$  has a
P2G-critical suffix of type $\{a,b\}$ that forms the first term of an
RRS for $w$. These properties are not affected by substituting $\zeta^\Rev$ for
$\zeta$, so the criticality type of $y_1'$ is the same as that of $u_1$.
But since $u_1$ has no proper critical suffixes of the same type, this
could only happen if $|y_1'| = |u_1|-1$, and that would give
This contradiction completes the verification of the optimality of $V$.

Since application of $V$ to $w'$ yields the same result as application of $U$
to $w$, we have $v=v'$ in part (ii)\,(c). 
\end{addmargin}

\smallskip
{\bf Case 3:} 
$m>0$, and the subword $\zeta$ starts at the beginning of $w_1$.

\begin{addmargin}[1em]{0em}
Let $w'_1$ be the suffix of $w_1$ of length $|w_1|-2$
(so that $w_1=\zeta w_1'$),  and define $y_1 := \f{\zeta}w_1'$,
$y_i:=w_i$ for $i>1$,
$\nu = \mu\l{\zeta}$, $\eta:=\gamma$, $v_1 := y_1$ and $v_i:= u_i$ for $i>1$.
Since $w_1=u_1$ is critical, and its second letter commutes with its
first, it follows from Definitions~\ref{def:P2G} and ~\ref{def:P3G}
that $v_1$ is critical of the same type as $u_1$.
Now, $\l{\zeta}=\l{\nu}$ forms part of $\alpha(u_1)$ that is not
within $\alpha(v_1)$,
so by definition we have $\tau(u_1)= \f{\zeta}\tau(v_1)$, and $V$ is an RRS.

Just as in Case 2, in order to prove that $V$ is an optimal RRS for $w'$ when
the hypotheses of (ii) hold, we need only to verify 
those conditions of Definition~\ref{def:optRRS} that refer to $y_1=v_1$;
we just need to check conditions (iii) and (i). Condition (iii) is again
clear, and the verification of Condition (i) is similar but more straightforward
to that in Case 2. And as in Case 2 we have $v=v'$ in part (ii)\,(c).

We observe that we have the situation of Case 2, but with the pairs $w$, $w'$,
$U$, $V$, and $\zeta$, $\zeta^\Rev$ interchanged.
\end{addmargin}

In each of the remaining cases, the first letter of the subword $\zeta$ of $w$
starts to the right of the first letter of $w_1$, and we shall see that the
decompositions $U$ and $V$ of $w$ and $w'$ start in the same positions in the
words $w$, $w'$. That is, the left-hand ends of $w_1$ and $v_1$ are in the same
positions.

This means that,
once we have established part (i) in all cases, as well as part (ii) in Case 2,
Condition (i) of the optimality of $V$ in part (ii)\,(b) can be
established in all remaining cases (4 onwards) by the argument of the following paragraph. 

If Condition (i) of Definition~\ref{def:optRRS} for optimality
failed for $V$, then there would be an optimal RRS $V'$ for $w'$ starting further to the
right in $w'$ than $V$.  
If we now perform the substitution $\zeta^\Rev \to \zeta$ on $w'$ and apply 
part (i) to $V'$, we obtain an RRS $U'$ for $w$; since the RRS $U$ for $w$ was 
assumed optimal, $U'$ cannnot start further right in $w$ than $U$ does, and 
hence  $U'$ must 
start further to the left in $w$ than $V'$ does in $w'$.
It follows that the application of part (i) of the proposition that derives the RRS $U'$ for $w$ from the RRS $V'$ for $w'$ is an instance of Case 2; 
and in that case optimality has already been proved, so $U'$ must be optimal
and hence (by Proposition~\ref{prop:unique_optRRS}) equal to $U$.
Now, from the arguments of Case 2, we see that we must have 
$u_1 = \f{y'_1}\f{\zeta^R}\s{y'_1}$ and hence the second letter of $\zeta$ is the second letter of $u_1$, from which we deduce that the first letter of $\zeta$ is the first letter of $u_1$.
But then our current substitution $\zeta \to \zeta^\Rev$ of $w$ would
have to be an instance of Case 3 of this proposition, which it is not.
So we have established a contradiction and hence Condition (i) of Definition~\ref{def:optRRS}  must hold for $V$.

\smallskip
{\bf Case 4:} $\zeta$ is contained within $w_i$ for some $i \le m$, but $\zeta$
does not include the first letter of $w_1$ when $i=1$.
In this case, $\zeta$ is within $u_i$ but starts after $\f{u_i}$.

{\it Case 4(a)}: $\zeta$ does not contain the last letter of $w_i$.

\begin{addmargin}[1em]{0em}
Either at least one of the letters of $\zeta$ is an internal letter of $u_i$,
or $u_i$ is of type $(a,b,c)$ 
	and the letters of $\zeta$ have names $a,c$.
	In the second case 
	(in the notation of Definition~\ref{def:P3G}),
	we cannot have $a^{\pm 1} $ in $\alpha(u_i)$, since in that case $u_i$ must start with $c^{\pm 1}$, and then 
	since $U$ must be optimal, by Lemma~\ref{lem:properties_optRRS}(ii)
	 $\f{u_i}$ would be the only letter of name $c$ in $(u_i)_p$ and hence the first letter of $\zeta.$
	So in that second case, $a^{\pm 1}$ is in the subword $(u_i)_r$ 
	and the other letter of $\zeta$ is an adjacent letter
$c^{\pm 1}$ (which could be $\l{(u_i)_q}$ or $\f{\alpha((u_i)_r))}$).

In either case,
we find a decomposition for $w'$ with $\nu:=\mu$, $\eta:=\gamma$,
$y_j:=w_j$, $v_j:=u_j$ for $j \neq i$, and $y_i$ formed from $w_i$ by
replacing its subword $\zeta$ by $\zeta^\Rev$.
That same commutation move transforms $u_i$ to a critical word $v_i$ and
either $\tau(u_i)=\tau(v_i)$, or the two letters of $\zeta$ are both
internal and we get $\tau(v_i)$ from $\tau(u_i)$ by replacing a subword
$\zeta$ by $\zeta^\Rev$. So we have a corresponding RRS for $w'$ with
$v_j=u_j$ for $j\neq i$.

Conditions (ii) and (iii) of Definition~\ref{def:optRRS} for optimality of $V$
in part (ii)\,(b) follow immediately from the optimality of $U$, and for
part (ii)\,(c) either $v=v'$ or $v'$ is obtained from $v$ by replacing
a subword $\zeta$ by $\zeta^\Rev$.
\end{addmargin}

\smallskip
{\it Case 4(b)}: $\zeta$ ends at the end of $w_i$. 

\begin{addmargin}[1em]{0em}
Then $\f{\zeta}$ is in $\beta(u_i)$ or $\rho(u_i)$, and in the former case
we have $i < m$ and $\alpha(u_{i+1})=\emptyword$ by Lemma~\ref{lem:RRSdetails}.
We define a decomposition for $w'$ with $\nu:=\mu$, $\eta:=\gamma$, $y_j:=w_j$
and $v_j:=u_j$ for $j \neq i,i+1$, but $y_i := w'_i\l{\zeta}$,
where $w'_i$ is the prefix of $w_i$ of length
$|w_i|-2$, and $y_{i+1}:= \f{\zeta}w_{i+1}$.
Then the word $v_i := u'_i\l{\zeta}$, where $u'_i$ is the prefix of $u_i$ of
length $|u_i|-2$, is critical, because it is derived from $u_i$ by deleting a
letter of $\beta(u_i)$ or $\rho(u_i)$, and we have
$\tau(u_i) = \tau(v_i)\f{\zeta}$ (when $\f{\zeta}$ is in $\beta(u_i)$) or $\tau(u_i) = \f{\zeta}\tau(v_i)$ (when $\f{\zeta}$ is in $\rho(u_i)$). 
We find that $v_{i+1} = u_{i+1}$ in the former case, and $v_{i+1}$ is equal
to $u_{i+1}$ with the letter $\f{\zeta}$ removed in the latter case.
In either case we have $v_j=u_j$ for $j \ne i,i+1$, and we can check that the
sequence $V=v_1,\ldots,v_{m+1}$ satisfies the conditions of
Definition~\ref{def:RRS} associated with the decomposition we have described.

Again Conditions (ii) and (iii) of the optimality of $V$ in part (ii)\,(b)
follow easily from the optimality of $U$.  In part (ii)\,(c) of the
proposition we have $v=v'$ when $\f{\zeta}$ is in  $\beta(u_i)$.
When $\f{\zeta}$ is in $\rho(u_i)$, it is pushed to the left when we apply
the RRS $U$ and ends up immediately to the left of the transformed word
$\tau(\widehat{u_i})$, whereas it ends up at the beginning of the transformed
word $\tau(\widehat{v_{i+1}})$ after applying the RRS $V$. So this letter is in
different positions in the words $v$ and $v'$ but the letters in between these
two positions all commute with $\f{\zeta}$, and so part (ii)\,(c) holds.
\end{addmargin}

\smallskip
{\bf Case 5:} $\zeta$ overlaps the end of $w_i$ and the beginning of $w_{i+1}$ 
for some $i < m$. 

\begin{addmargin}[1em]{0em}
(This reverses the transformation in Case 4(b).)
We define a decomposition of $w'$ with $\nu:=\mu$, $\eta := \gamma$,
$y_j := w_j$ and $v_j := u_j$ for $j \ne i,i+1$,
$y_i := \p{w_i}\zeta^\Rev$ and $y_{i+1}:= \s{w_{i+1}}$.
Then $v_i := \p{u_i}\zeta^\Rev$ is critical with
$\tau(v_i) = \tau(u_i)\f{\zeta}$ or $\f{\zeta}\tau(u_i)$ (when $\f{\zeta}$ is
in $\beta(v_i)$ or $\rho(v_i)$ respectively), and $v_{i+1}=u_{i+1}$ in
the former case.  We can check that the sequence $V=v_1,\ldots,v_{m+1}$
satisfies the conditions of Definition~\ref{def:RRS} associated with the
decomposition we have described.  Part (ii) is again straightforward,
although if $\l{\zeta}$ is in $\rho(v_i)$ then this letter ends up in
in different positions in the words $v$ and $v'$ in part (ii)\,(c), with
letters that commute with it in between these two positions.
\end{addmargin}

\smallskip
{\bf Case 6:} the end of $\zeta$ is within $w_{m+1}$. 

\begin{addmargin}[1em]{0em}
If $m>0$, or $m=0$ and $\f{\zeta}$ is not the first letter of $u_1=w_1$,
then $\zeta$ lies in $\s{u_{m+1}}$, and consists of letters
that commute with $\f{u_{m+1}}$. The RRS $V$ is defined by replacing
$w_{m+1}$ and $u_{m+1}$ by the results of replacing their subwords $\zeta$
by $\zeta^\Rev$, and the required properties of $V$ are easily seen to hold.

Otherwise $m=0$ and $\zeta$ consists of the first two letters of $u_1=w_1$,
and the associated decomposition of $w'$ is
$\mu \l{\zeta} \f{\zeta} w_1' \gamma$ with $v_1=y_1 := \f{\zeta} w_1'$,
where $w_1'$ is the suffix of $w_1$ of length $|w_1|-2$.
Part (ii) is straightforward in both cases, with $v=v'$ in part (ii)\,(c).
\end{addmargin}

In the remaining two cases $\zeta$ intersects $\gamma$ non-trivially, so the
hypotheses of part (ii) of the proposition do not hold.

\smallskip
{\bf Case 7:} the end of $\zeta$ is at the beginning of $\gamma$.

\begin{addmargin}[1em]{0em}
Suppose first that $m>0$.  We claim that $\l{w_m}$ does not commute
with $\f{\gamma}$.
To see this, note that by Lemma~\ref{lem:RRSdetails}\,(i) $w_m$ is 2-generator
critical with $\beta_m = \emptyword$. So after applying the first $m$ steps of
the RRS, the new letter at the end of $w_m$ (which becomes $\f{u_{m+1}}$) is
the other critical generator of $u_m$, and so it does not commute with the
original $\l{w_m}$.  But $\f{u_{m+1}} =\f{\gamma}^{-1}$, so $\l{w_m}$ does not
commute with $\f{\gamma}$.  So $\f{\zeta}$ is in $w_{m+1}$ in this case,
and we find a decomposition of $w'$ with $\nu:= \mu$, $y_i:= w_i$ and
$v_i:=u_i$ for $i \leq m$, $y_{m+1} := \p{w_{m+1}}$
and $v_{m+1} := \p{u_{m+1}}$, and $\eta := \zeta^\Rev\s{\gamma}$.
The case $m=0$ is also straightforward,
\end{addmargin}

\smallskip
{\bf Case 8:} the beginnings of $\zeta$ and $\gamma$ coincide.

\begin{addmargin}[1em]{0em}
(This reverses the transformation in Case 7.) If $m>0$ we find a decomposition
of $w'$ with $\nu:= \mu$, $y_i:= w_i$ and $v_i:=u_i$ for $i \leq m$,
$y_{m+1} := w_{m+1}\l{\zeta}$ and $v_{m+1} := u_{m+1}\l{\zeta}$, and
$\eta := \f{\zeta}\gamma'$, where $\gamma'$ is the suffix of $\gamma$ of
length $|\gamma|-2$.
The case $m=0$ is again straightforward,
\end{addmargin}
\end{proof}

\begin{proposition}\label{lem:6.2}
Suppose that the word $w \in A^*$ admits an optimal RRS
$U=u_1,\ldots,u_m,u_{m+1}$
with decomposition $w=\mu w_1 \cdots w_mw_{m+1}\gamma$.
	Suppose also that $w$ has a $2$-generator $\{s,t\}$-critical subword $\zeta$, and 
let $w'$ be the word obtained from
$w$ by replacing the subword $\zeta$ by $\tau(\zeta)$.
Then:
\begin{mylist}
\item[(i)] The word $w'$ admits an RRS of length $m$;
we denote that by $V$, and the associated decomposition by
$w' = \nu y_1 \cdots y_{m+1}\eta$.
\item[(ii)] Suppose that $\p{w} \in W$, so that $\gamma$ is a single letter
$\xx$, and that $\zeta$ is within $\p{w}$.
Then the RRS $V$ can be chosen to ensure that:
  \begin{mylist}
  \item[(a)] $\eta=\gamma=\xx$;
	  \quad{(b)}\quad $V$ is optimal;
  \item[(c)]  if $w \rightarrow_U vx^{-1}x$ and $w \rightarrow_V v'x^{-1}x$,
               then $v$ can be transformed to
               $v'$ using only commutation relations of $G$ and 2-generator $\tau$-moves
   \end{mylist}
\end{mylist}
\end{proposition}

\begin{proof}
By essentially the same argument as we used at the beginning of the proof of
Proposition~\ref{lem:6.1}, it suffices to prove parts (i), (ii)\,(b) and
(ii)\,(c).

Aiming for greater readability, 
we have divided this proof into two parts.
Part 1 describes the 
construction of a sequence $V$ associated with a decomposition of $w'$ that we claim is an optimal RRS;
as in the proof of Proposition~\ref{lem:6.1}, we follow
a list of cases according to the position of the subword $\zeta$ within $w$.
	We leave it to the reader to verify in each case that $V$ is an RRS,
	and hence that (i) holds,
as well as (ii)\,(b) and (ii)\,(c) in Cases 1,7 and 8.
Part 2 of the proof contains technical details that verify claims made in part 1 as well
 as details of the proofs of
	 (ii)\,(b) and (ii)\,(c) in (most of) the cases 2--6. We consider that part 2 could be ignored on a preliminary reading.  

	\subsubsection*{Part 1: construction of $V$:}
\smallskip
{\bf Case 1:}  $\zeta$ is within $\mu$ or $\s{\gamma}$.

\begin{addmargin}[1em]{0em}
As in Proposition~\ref{lem:6.1}, 
we define $V:=U$, 
associated with a decomposition of $w'$ of the form $\nu w_1\cdots w_{m+1}\eta$.
\end{addmargin}

\smallskip
{\bf Case 2:} The end of $\zeta$ (i.e. its last letter) is at the beginning of
$w_1$ (its first letter), and either 
\begin{mylist}
\item[(a)] $m=0$; or $m>0$ and either 
\item[(b)] $w_1$ is a P2G-critical word exactly one of whose pseudo-generators
is $s$ or $t$, or $w_1$ is a critical word of type $(a,b,c)$ and
$|\{s,t\} \cap \{b,c\}|=1$; or 
\item[(c)] $\alpha(u_1)$ does not commute with both $s$ and $t$.
\end{mylist}

\begin{addmargin}[1em]{0em}
In all of these cases we define $V$ of length $m+1$, with $v_1=y_1$ equal to the
shortest critical suffix of $\tau(\zeta)$, and $y_2:=\s{w_1}$,
$y_i:=w_{i-1}$ for $i>2$, and $v_i := u_{i-1}$ for $i>1$.
\end{addmargin}

\smallskip
{\bf Case 3:} $w_1$ begins within $\zeta$, and $\zeta$ ends within $w_1$, 
and either
\begin{mylist}
\item[(a)] $\zeta$ ends at the beginning of $w_1$ and we are not in Case 2,
so all three of the following hold:
(1) $m>0$; (2) $w_1$ is either P2G-critical with pseudo-generators $\{s,t\}$
or critical of type $(a,b,c)$ with $\{b,c\}=\{s,t\}$; and
(3) $\alpha(u_1)$ commutes with both $s$ and $t$; or
\item[(b)] the intersection $\zeta'$ of the two subwords $\zeta$ and $w_1$
within $w$ has length greater than one.
\end{mylist}

\begin{addmargin}[1em]{0em}
Suppose first that $\zeta$ ends at the end of $w_1$ (this can only happen
within  Case 3(b)).  Then we define $V$ of length $m-1$, with
$y_1 := \l{\tau(\zeta)} w_2$, $y_i:=w_{i+1}$ for $i>1$, and $v_i:=u_{i+1}$
for $i \ge 1$. This reverses the transformation in Case~2.

Otherwise (we could be in Case 3(a) or 3(b))
$w_1$ has a non-empty suffix $w_{1s}$ that is not part of $\zeta$. In
that case we define $V$ of length $m$, with $y_1=v_1$ the shortest critical suffix
of $\tau(\zeta) w_{1s}$ that has the same criticality type as $u_1$,
and $y_i:=w_i$, $v_i :=u_i$ for $i>1$. 
\end{addmargin}

In each of the remaining Cases 4--8 we define $V$ to have the same length as $U$, with the
beginnings of $w_1$ and $y_1$ in the same positions in $w$ and $w'$.

\smallskip
{\bf Case 4:} $\zeta$ is a subword of $w_i$ for some $1 \le i \le m$,
but does not start at the beginning of $w_1$ (when $i=1$).

\begin{addmargin}[1em]{0em}
If $s$ and $t$ are both internal letters of $u_i$, then we define $v_i$ and
$y_i$ to be the words $u_i$ and $w_i$ with the subword $\zeta$ replaced by
$\tau(\zeta)$, and $v_j:=u_j$, $y_j:=w_j$ for $j \ne i$.

We exclude the possibility that exactly one of $s,t$ is an internal letter of
$u_i$ in Case 4 of Part 2.
It remains to deal with the case where both $s$ and $t$ are pseudo-generators of
	$u_i$.  In this case, as is verified in Case 4 of Part 2 below,  
$u_i$ has type $(a,b,c)$
and $\zeta$ is an $\{a,b\}$-critical subword of $(u_i)_q(u_i)_r$ in the
notation of Definition~\ref{def:P3G}.
So the intersection of $(u_i)_q$ and $\zeta$ is a possibly empty suffix
of $(u_i)_q$, and equal to a power of $b$.
	Recall from Definition~\ref{def:P3G} that $\tau(\widehat{(u_i)_r})$
	has the form
$b^{\ii}a^{\jj}b^{\kk}$.

	In the case where $\zeta=(u_i)_r$ and so 
$\tau(\zeta) = b^{\ii} a^{\jj} b^{\kk} = \tau((u_i)_r)=\tau(\widehat{(u_i)_r})$, then,
where $w_i=w_{ip} \zeta$ and $u_i = u_{ip} \zeta$,
we define $y_i:= w_{ip}b^{\ii}$, $v_i := u_{ip}b^{\ii}$,
$y_{i+1} := a^{\jj} b^{\kk} w_{i+1}$, $v_{i+1}:=u_{i+1}$, and
$y_j:=w_j$, $v_j:=u_j$ for $j \ne i,i+1$.
We see that the component $v_i$ of $V$  is P2G-critical of
type $\{b,c\}$.

If  $\zeta \neq (u_i)_r$ and $\zeta$ ends before
the end of $(u_i)_r$,
then we define $y_i$ and $v_i$ to be the words we derive from $w_i$ and 
$u_i$ respectively
by replacing their subwords $\zeta$ by $\tau(\zeta)$; the resulting words still end in letters with name $a$,
and we define $y_j:=w_j$, $v_j := u_j$ for $j \ne i$.

Finally, if $\zeta$ contains $(u_i)_r$ as a proper suffix then
	the word derived from $w_i$ by replacement of its subword $\zeta$ by $\tau(\zeta)$
	ends with a nonzero power $b^t$
of $b$, which becomes a prefix of $y_{i+1}$ in the decomposition $V$.
More precisely, writing $w_i = w_{ip} \zeta$, $u_i = u_{ip} \zeta$, and $\tau(\zeta) = \zeta' b^t$ we have $y_i := w_{ip} \zeta'$,
$v_i := u_{ip} \zeta'$, $y_{i+1} := b^t w_i$, $v_{i+1} := u_{i+1}$,
and $y_j:=w_j$, $u_j := v_j$ for $j \ne i,i+1$.
\end{addmargin}

\smallskip
{\bf Case 5:} $\zeta$ starts within $w_i$ and $\zeta$ intersects
$w_{i+1}$ non-trivially for some $1 \le i < m$.

\begin{addmargin}[1em]{0em}
Using the optimality of $U$ and $A_3$, $B_3$-free assumption, we prove
in Case 5 of Part 2 below that there are essentially only two different possibilities that we
need consider for the criticality types of $u_i$ and $u_{i+1}$.

	{\bf Subcase 5\,(a)}
In this first of the two subcases, $u_i$ has type $(a,b,c)$ and $u_{i+1}$ has type
$\{b,c\}$ or $(a,b,c)$, and $\zeta$ is an $\{a,b\}$-word with
$\zeta = \zeta' b^t$ for some suffix $\zeta'$ of $(u_i)_r$ and prefix $b^t$ of
$w_{i+1}$.  Writing $w_{i} = w_i' \zeta'$, $u_i = u_i' \zeta'$, and
$w_{i+1} = b^t \pi$ for prefixes $w_i'$ and $u_i'$ of $w_i$ and $u_i$, and
suffix $\pi$ of $w_{i+1}$, the word $w'$ has an RRS $V$ with
$y_{i} := w_i' \tau(\zeta)$, $v_i := u_i' \tau(\zeta)$, $y_{i+1} := \pi$,
$v_{i+1} := u_{i+1}$, and $y_j:=w_j$, $v_j := u_j$ for $j \ne i,i+1$, where
	$\tau(\widehat{(v_{i})_r})=b^{\ii}a^{\jj}b^{\kk+t}$ 
(note that $\kk$ and $t$ must have the same signs).
This reverses the transformation that we described in the final situation of
Case 4, and applying $V$ to $w'$ gives the same result as applying $U$ to $w$.

	{\bf Subcase 5\,(b)} The second possibility is that $u_i$ has type $\{r,t\}$ for some $r \ne s$
and with $m_{rt} \ge 5$, and $u_{i+1}$ has type $\{r,t\}$, $(a,r,t)$
or $(a,t,r)$ for some generator $a$. We have $m_{rs}=2$ and
$m_{st} = 3$, and $\zeta = t^{\ii}s^{\jj}t^{\kk}$ for some nonzero
$\ii,\jj,\kk$, where $w_{i} = w_i' t^{\ii}$, $u_i = u_i' t^{\ii}$ and
$w_{i+1} = s^{\jj}t^{\kk} \pi$ for prefixes $w_i'$, $u_i'$ of
$w_i$ and $u_i$ and suffix $\pi$ of $w_{i+1}$.
Then the word $w'$ has an RRS $V$ with $y_{i} := w_i' \tau(\zeta)$,
$v_i := u_i' \tau(\zeta)$, $y_{i+1} := \pi$, $v_{i+1} := u_{i+1}$ and
$y_j:=w_j$, $v_j:=u_j$ for $j \ne i,i+1$, where $v_i$ has type $(r,s,t)$.
This reverses the transformation that we described in the first situation
of Case 4, and again applying $V$ to $w'$ gives the same result as applying
$U$ to $w$.

Note also that it was an application of the transformation $w \to w'$ in this
situation that resulted in Example~\ref{eg:tricky_w_in_G} and forced us to
introduce critical words of type $(a,b,c)$.
\end{addmargin}

\smallskip
{\bf Case 6:} $\zeta$ ends within $w_{m+1}$ but not in the first letter of
$w_1$ when $m=0$ (which was considered in Case 2). 

\begin{addmargin}[1em]{0em}
We prove in Part 2 below that this case can only occur when $\zeta$ lies
entirely within $w_{m+1}$ and $u_{m+1}$, in which case we define $y_{m+1}$
and $v_{m+1}$ to be derived from $w_{m+1}$ and $u_{m+1}$ by replacing $\zeta$
by $\tau(\zeta)$, and the required properties of $V$ are clear.
\end{addmargin}

In the final two cases $\zeta$ intersects $\gamma$ non-trivially, so we only
have to prove part (i) of the proposition.

\smallskip
{\bf Case 7:} $\zeta$ intersects both $\gamma$ and $w_mw_{m+1}$ non-trivially.

\begin{addmargin}[1em]{0em}
Since $w_{m+1}$ commutes with $\f{\gamma}$, we must have $w_{m+1}=\emptyword$.
The beginning of $\zeta$ must lie in $w_m$, or we would be in Case 5
(in which case $\zeta$ cannot intersect $\gamma$), and $u_m$ must be a P2G
word by Lemma~\ref{lem:RRSdetails}\,(i). In fact, since the two
pseudo-generators of $u_m$ are the names of $\l{w_m}$ and $\f{\gamma}$, these
must be the generators of $\zeta$.

After applying the first $m-1$ steps of the RRS $U$ to $w$, the resulting word
has the non-geodesic P2G subword $u_m\f{\gamma}$. After replacing $\zeta$ by
$\tau(\zeta)$, we still have a non-geodesic P2G subword starting in the same
place, so we can continue this RRS to get an RRS $V$ of $w'$ (of which the
right-hand end could be to the left of that of the RRS $U$).
\end{addmargin}

\smallskip
{\bf Case 8:} $\zeta$ is a subword of $\gamma$. 

\begin{addmargin}[1em]{0em}
By Case 1 we only need consider the case when 
$\zeta$ is a prefix of $\gamma$. Let $t := \f{\gamma}=\f{\zeta}$.
We apply the first $m$ moves in the RRS $U$ to obtain a word ending
in $t^{-1} w_{m+1} \gamma$. Then $t^{-1} w_{m+1} \zeta$ is a non-geodesic P2G
word, and hence so is $t^{-1} w_{m+1} \tau(\zeta)$, and we can continue our
RRS to obtain an RRS $V$ of $w'$ (of which the 
right-hand end could be to the right of that of the RRS $U$).
\end{addmargin}

\subsubsection*{Part 2: further technical details}

{\bf Case 2:}

\begin{addmargin}[1em]{0em}
{\bf Proof of (ii)\,(c):}
	We have
$\tau(\zeta) = \rho y_1$ for some word $\rho$ and 
obtain the word $v'$ from $v$ by replacing the subword $\p{\zeta}$ of $v$ by
$\rho \p{\tau(y_1)}$. Since the two subwords are geodesics representing the same
element in a $2$-generator Artin group, by Lemma~\ref{lem:2gengeo} we can
obtain one from the other using $\tau$-moves on 2-generator subwords.

{\bf Proof of (ii)\,(b):}
	First suppose that we are in one of the two Subcases 2\,(b) or 2\,(c)
	(when $m>0$). 
Since the suffixes of $w$ and $w'$ to the right of the last letter of $y_1$
are the same, application of Procedure~\ref{proc:unique_optRRS} to $w'$ locates
the words $y_{m+2}=u_{m+1}$, $y_{m+1}=u_m$, \ldots, $y_3=u_2$ of $V'$,
as in the decomposition of $V$, where $v_2$ has the same type as $u_1=w_1$.
(The argument that these types are the same is similar to that in Case 2
of Proposition~\ref{lem:6.1}: the type is determined by a proper suffix of
$v_2$ and is not affected by a change of its first letter.)

	The conditions on Subcases 2\,(b) and 2\,(c)
	ensure that $y_1y_2$ has no critical suffix
of the same type as $u_2$, and by considering
Procedure~\ref{proc:unique_optRRS}, we see it produces the RRS $V$ of length
$m+1$ exactly as we have defined it.  So $V$ is optimal as claimed.

	If $m=0$ (that is in subcase 2(a)), it is easy to see that condition (i) of Definition~\ref{def:optRRS} holds. Then condition (ii) of that definition is
	inherited by $V$ from $U$, while condition (iii) holds vacuously, since in subcase 2\,(a) the hypotheses of condition (iii) do not hold.
\end{addmargin}

{\bf Case 3:}

\begin{addmargin}[1em]{0em}
	Suppose first that we are in Subcase 3\,(b).
	Note that the first two letters of $w_1$ have
distinct names by Lemma~\ref{lem:properties_optRRS}\,(ii).
Since the two generators $s,t$ of $\zeta$ do not commute, we satisfy
	condition (1) of Subcase 3\,(a) that $m>0$.
Further, since $s,t$ are the first two letters of the critical word $w_1=u_1$,
it follows from Definitions~\ref{def:P2Gabr} and~\ref{def:P3Gabr} that
	$\alpha(u_1) = \emptyword$ and so condition (3) from Subcase 3\,(a) holds,
	and then  (from those definitions) condition (2) from Subcase 3\,(a) is immediate. 
	It follows that from now on, which ever of the two subcases we are in, we may assume that all three conditions (1)--(3) from Subcase 3\,(a) hold
	(but note that these conditions alone do not define Subcase 3\,(a)).

	{\bf Proof of (ii)\,(b):}
	If $\zeta$ ends at the end of $w_1$ (subcase 3(a)), then
$\beta(u_1) = \emptyword$ and $w_1$ is a $2$-generator critical word
with $\l{\tau(\zeta)} = \l{\tau(w_1)}=\l{\tau(u_1)}$.
So we find the critical word $u_2$ as a subword of $w'$,
starting at the end of the subword $\tau(\zeta)$ of $w'$, and we
define $y_1$ to be that subword.
Since $U$ is optimal, optimality for $V$ could only fail if
$y_1$ had a proper critical suffix or if condition (iii) of
Definition~\ref{def:optRRS} failed for $i=2$. But a proper critical suffix of $y_1$
would also be a proper critical suffix of $u_2$, and so within $w_2$, and hence
contradict the optimality of $U$, so we need only consider the second possibility.  

	In Subcase 3\,(b) we have $w_1 = \zeta' w_{1s}$ for some suffix $\zeta'$ of $\zeta$
of length at least 2 and a non-empty suffix $w_{1s}$ of $w_1$, where the word
$\zeta w_{1s}$ is critical of the same type as $w_1$.
Since $U$ is optimal, in  this subcase too  we only have to verify the optimality 
Definition~\ref{def:optRRS}\,(i) for $V$ together with (iii) for $i=2$ in order to
complete the proof of
(ii)\,(b).
Since $w$ and $w'$ differ only to the left of their common 
suffix $w_2\cdots w_{m+1}\gamma$, Procedure~\ref{proc:unique_optRRS}
will locate the same subwords $y_{m+1},\ldots,y_2$ in $w'$ as
the subwords $w_{m+1},\ldots,w_2$ of $w$,
the same location of the right-hand end of $y_1$ in $w'$ as that of $w_1$ in
$w$, and we find that the type of $v_1$ is the same as that of $u_1$.
That condition (iii) of Definition~\ref{def:optRRS} also holds, even for $i=2$,
is an immediate consequence of the optimality of $U$.
For condition (i) of optimality, Procedure~\ref{proc:unique_optRRS} will
define $y_1=v_1$ exactly as in our definition of $V$, so $V$ is optimal.

	{\bf Proof of (ii)\,(c):}
In 
	both subcases, 
	the words $v$ and $v'$ in part (ii)\,(c) of the
proposition differ only in $2$-generator geodesic subwords that represent
the same element in the $2$-generator Artin group generated by $s$ and $t$,
and so part (ii)\,(c) holds by Lemma~\ref{lem:2gengeo}.
\end{addmargin}

In each of the remaining cases 4--8, the decompositions $U$ and
$V$ of $w$ and $w'$ start in the same position in the words $w$, $w'$. That
is, the left-hand ends of $w_1$ and $v_1$ are in the same position.
By a similar argument that we used in Proposition~\ref{lem:6.1}, this implies
that condition (i) of the optimality of $V$ holds in part (ii)\,(b)
in these cases.

{\bf Case 4:}
\begin{addmargin}[1em]{0em}
Recall that in this case $\zeta$ is a subword of $u_i$.

	{\bf Proof of (ii)\,(b) and (c) when $s,t$ are both internal:}
Suppose that $s$ and $t$ are both internal letters of $u_i$. 

	First we show that $\zeta$ does not intersect $\beta(u_i)$. 
	For if it did, since it contains only internal letters it would be a (critical) subword of $(u_i)_s$ or $((u_i)_r)_s$  (depending on whether $u_i$ is P2G or P3G).
	But, by Lemma~\ref{lem:RRSdetails}\,(ii) and (iii),
the word $\beta(u_i)$ is either a power of a generator or of the form $b^Ie^J$
for generators $b$ and $e$.
In the second case, we cannot have
	$\{s,t\} = \{b,e\}$, since in that case $\zeta$ would need to be a subword of $\beta$, but $\beta$ certainly has no critical subword.
	So in either case, the intersection of $\zeta$ with $\beta(u_i)$ would
	be a non-trivial power of a single generator, which we may assume to be $s$, and any $s$ in $\zeta$ must be within $\beta$. Then the instances of $t$ within $\zeta$ must be in $\rho(u_i)$, and at least one of those must be to the right of an $s$ within $\zeta$; such a $t$ would have to commute with $s$ (by definition of $\rho$), but cannot since $\zeta$ is critical.

	So $\zeta$ does not intersect $\beta$.
	Then by (Definitions~\ref{def:P2G},\ref{def:P3G}) the whole word $\zeta$ is
pushed to the left of the word when we apply $\tau$ to $u_i$, and we see
	easily that the RRS $V$ that we defined for the word $w'$ is optimal, and (ii)\,(b) is proved.
For (ii)\,(c) in this situation, we note that the word $v'$ is obtained from $v$ by replacing a
	subword $\zeta$ by $\tau(\zeta)$. 

	{\bf Excluding the possibility that just one of $s,t$ is internal:}
We exclude the possibility that exactly one of $s,t$ is an internal letter of
$u_i$ as follows.

Suppose (without loss of generality) that $s$ is an internal letter of 
$u_i$ but that $t$ is not. Since $s$ and $t$ do not commute but each 
occurrence of $s$ within $\zeta$ must be in one of $\alpha_i,\rho_i$ or $\beta_i$, it follows that any letter with name $s$ within $\zeta$ must be either
	to the left or to the right of all powers of $t$ within $\zeta$, and also within $u_i$. So we have $\zeta = s^jt^ks^l$ for some $j,k,l\neq 0$ (so $s_j \in \alpha_i,s_l \in \beta_i$), and hence $m_{st}=3$. 

Now $u_i$ is either P2G-critical, with pseudogenerators $\uu \neq s,t$, 
	or P3G-critical of type $(a,b,c)$.
In the first case, the same argument as above shows that the 2-generator 
critical word $\widehat{u_i}$ must have the form ${\uu}^{j'}t^k{\uu}^{l'}$ for 
some $j',l' \neq 0$, and hence $m_{\uu,t}=3$, $m_{s,\uu}=2$, and we have 
a $(2,3,3)$ triangle ($A_3$ subdiagram) in our Coxeter diagram for $G$, contradicting our hypotheses.
In the second case, the syllable $s^l$ of $\zeta$ must lie in $\beta(u_i) = \beta((u_i)_r)$, so $s$ commutes with $\l{u_i}$, which has name $a$. 
So $t = b$ or $c$.
	If $t=b$ then $m_{ab}=m_{sb}=3$, $m_{sa}=2$, and so we have an
	$A_3$ subdiagram in our Coxeter diagram, contradicting our hypotheses, 
	while if $t=2$ then, since $a$ and $c$ commute, the whole of $\zeta$ lies 
	in $\beta(u_i)$, and so $\beta(u_i)$ has at least three syllables, 
	contradicting Proposition~\ref{lem:RRSdetails}\,(iii).

	{\bf The case where $s,t$ are both pseudo-generators:}
	It remains to consider the possibility
that $s$ and $t$ are both pseudo-generators of $u_i$. In that case it
follows from Lemmas~\ref{lem:critsubword} and~\ref{lem:properties_optRRS}\,(i)
that $u_i$ cannot be a P2G-critical word and if $u_i$ is critical of type
$(a,b,c)$ then $\zeta$ cannot be a critical $\{b,c\}$-subword of that.
So $u_i$ must be critical of type $(a,b,c)$ with
$\zeta$ a critical $\{a,b\}$-subword of the suffix $(u_i)_q{(u_i)}_r$ of $u_i$
in the notation of Definition~\ref{def:P3G}.

	{\bf Proving (ii)\,(b) when $s,t$ are both pseudo-generators:}
In the case when $\zeta = {(u_i)}_r$ and
$\tau(\zeta) = b^{\ii} a^{\jj} b^{\kk}$,
we see that $v_i$ has type $\{b,c\}$ rather than $(a,b,c)$, whereas
$v_{i+1}$ has the same type as $u_{i+1}$, which is $\{b,c\}$ or $(a,b,c)$.
We see easily that optimality conditions (ii) and (iii) still hold for $V$
	in part (ii)\,(b) of the proposition. (Recall that condition (i) does not need to be proved in cases 4--8.)

Otherwise, if $\zeta$ ends before the end of $(u_i)_r$, then, after replacing
$\zeta$ by $\tau(\zeta)$ in $u_i$, the resulting word is still critical of
type $(a,b,c)$, and we get $y_i$ and $v_i$ by making this replacement.

Finally, if $\zeta$ ends at the end of $(u_i)_r$ then, after replacing $\zeta$
in $w_i$ by $\tau(\zeta)$, the resulting word ends with a nonzero power $b^t$
of $b$, which becomes a prefix of $y_{i+1}$ in the decomposition $V$.
Again it is easily seen that optimality condition (ii) and (iii) hold for $V$
in part (ii)\,(b) of the proposition.

{\bf Proving (ii)\,(c) when $s,t$ are both pseudo-generators:}
Applying $V$ to $w'$ gives the same
result as applying $U$ to $w$, and so $v'=v$ in part (ii)\,(c) of the
proposition.
\end{addmargin}

{\bf Case 5:}

{\bf Proof that only two criticality types for $u_i$ are possible (as in subcases 5\,(a) and 5\,(b))}.

\begin{addmargin}[1em]{0em}
	Suppose first that $u_i$ is critical of type $(a,b,c)$.
Then $u_{i+1}$ is critical of type $\{b,c\}$ or $(a',b,c)$ for some generator
$a'$ (which may or may not be equal to $a$) by
Lemma~\ref{lem:RRSdetails}\,(iii). By definition of type $(a,b,c)$, we
have $\l{u_{i}} = \l{w_{i}} = a^{\pm 1}$, and $\zeta$ must be an
$\{a,b\}$-word with $\beta_i$ (which is a power of $c$) empty.
We have $\zeta = \zeta' b^t$ for some suffix $\zeta'$ of $(u_i)_r$ and
non-empty prefix $b^t$ of $w_{i+1}$, and $u_{i+1}=c^{\pm 1}b^{\kk}w_{i+1}$,
and we are in 
	Subcase 5\,(a) 
	of Case 5.
Checking the optimality of $V$ in part (ii)\,(b) of the proposition is straightforward in this situation.

Otherwise $u_i$ is a P2G-critical word. Then $\l{u_i}$ has name $s$ or
$t$, and we assume without loss that it is $t$.  So $\f{w_{i+1}}$ has name $s$.
If the intersection of $u_i$ with $\zeta$ contained both $s$ and $t$, then
$u_i$ would be critical of type $\{s,t\}$, and then $\l{\tau(u_i)}$ would
have name $s$ and we would contradict optimality of $U$ by
Lemma~\ref{lem:properties_optRRS}\,(ii).

So $u_i$ and $\zeta$ intersect in a power of $t$.
By Lemma~\ref{lem:properties_optRRS}\,(ii) again, $u_i$ cannot be critical
of type $\{s,t\}$, so it is critical of type $\{r,t\}$ for some $r \ne s$, and
the power of $s$ at the beginning of $w_{i+1}$ must lie in $\alpha_{i+1}$, and
so $m_{rs}=2$. Since this power of $s$ is followed by a letter with name $t$,
$u_{i+1}$ must have criticality type $\{r,t\}$, $(a,r,t)$, $(a,t,r)$ for
some generator $a$.

If we had $m_{st} > 3$, then $w_{i+1}$ would have a prefix of the form
$s^\jj t^\kk s^\ll $ for non-zero integers $\jj,\kk,\ll$, and so $u_{i+1}$
would have a prefix of the form $r^\ii s^\jj t^\kk s^\ll$ with $\ii \ne 0$.
This could happen only when $m_{rt}=3$ and $m_{st}=4$ with $t^L$ in
$\beta_{i+1}$, which contradicts our hypothesis that the Coxeter diagram has no
$(2,3,4)$-triangles.  (But note that Example~\ref{eg:n_atleast5} shows that
this situation can occur within such a triangle.)

So $m_{st}=3$, and our hypothesis that there are no $(2,3,3)$- or
$(2,3,4)$-triangles implies that $m_{rt} \ge 5$, and we are in 
	Subcase 5\,(b) of
	of Case 5.
Checking the optimality of $V$ again straightforward.
\end{addmargin}

{\bf Case 6:}

\begin{addmargin}[1em]{0em}
	{\bf Proof that $\zeta$ lies entirely within $w_{i+1}$ and $u_{i+1}$:}
Assume for a contradiction that this does not happen. 
	When $m=0$ and $|w_1| \ge 2$ the first two letters of $w_{m+1}$
are distinct and commute, so this cannot occur when $m=0$.

So $m > 0$ and the final letter of $w_m$, which we can assume to have name $t$,
is contained in $\zeta$. So $s$ commutes with but (by optimality of $U$) is
not equal to the other pseudo-generator $r$ of $u_m$. Since $\beta_m$ is empty
by Lemma~\ref{lem:RRSdetails}\,(i),
$\zeta$ and $w_m$ must intersect in a power of $t$, and there must be a
occurrence of a letter with name $t$ in the intersection of $w_{m+1}$ and
$\zeta$. But this letter would have to commute with $\f{u_{m+1}}$, which
has name $r$, so we have a contradiction.
\end{addmargin}
\end{proof}

\begin{proposition}\label{lem:6.3}
	Let $e,f \in S$ be commuting generators. Then
$[wef] = [wfe]$ for all $w \in W$.
\end{proposition}
\begin{proof}
Suppose first that $we \in W$. Then $wef \not\in W$ if and only if $wef$ admits
an RRS the application of which replaces the final $e$ of $we$ by $f^{-1}$,
in which case $w$ is equivalent to a word of form $w' f^{-1}$.
But that is the case if and only if $wf \not \in W$, in which case $[wef]=[wfe]
=[w'e]$. Otherwise $wef$ and $we$ are both in $W$, in which case so is $wfe$
by Proposition~\ref{lem:6.1} and, since $wef$ and $wfe$ are equivalent words
in $W$, we again have $[wef]=[wfe]$. The proof when $wf \in W$ is analogous.

It remains to deal with the case when neither $we$ no $wf$ is in $W$.
Then $w$ is equivalent to a word of form $w' e^{-1}$ and, by Propositions~\ref{lem:6.1} and \ref{lem:6.2} we have that
$w' e^{-1} f \not\in W$, so
$w'$ is equivalent to a word of form $w'' f^{-1}$. Then $[wef]=[w'']$. On the other hand, by the aforementioned propositions $[wfe]=[w'e^{-1}fe]$. Finally, since $w'e^{-1}\in W$ and $e^{-1}$ commutes with $f$, the previous paragraph implies $[w'e^{-1}fe]=[w''e^{-1}e]=[w'']$.
\end{proof}

\begin{proposition}\label{lem:6.4}
Let $s,t \in S$ be a non-commuting pair of generators with
${}_n(s,t) = {}_n(t,s)$.
Then $[w{}\cdot{}_n(s,t)] = [w{}\cdot{}_n(t,s)]$ for all $w \in W$.
\end{proposition}
\begin{proof}
We start by replacing $w$ by an equivalent word with suffix a longest possible
$\{s,t\}$-word. So $w = w' u$, where $u$ is an $\{s,t\}$-word, and $w'$ is not
equivalent to any word ending in a letter with name $s$ or $t$.

We claim that all reductions resulting from applications of RRS's when
computing $[w{}\cdot{}_n(s,t)]$ result from RRS's of length $1$ consisting of a
single $2$-generator critical $\{s,t\}$-word that does not intersect the prefix
$w'$ of the word, and similarly for $[w{}\cdot{}_n(t,s)]$.
Lemma~\ref{lem:2gengeo} will then imply that $[w{}\cdot{}_n(s,t)]$ and
$[w{}\cdot{}_n(t,s)]$ are both equal to $[w'v]$, where $v$ is a geodesic
representative of $u{}\cdot{}_n(s,t) =_G u{}\cdot{}_n(t,s)$ in the
$2$-generator Artin group $\langle s,t \mid {}_n(s,t)={}_n(t,s) \rangle$.

To prove the claim, suppose not. Then there are words $w'v$ and $w'vx$ with
$v$ an $\{s,t\}$-word, $x \in \{s,s^{-1},t,t^{-1}\}$, $w'v \in W$,
and $w'vx \not\in W$, but $vx \in W$. Let $U$ be the optimal RRS of $w'vx$,
which replaces $w'v$ by a word ending in $x^{-1}$. If $v$ is empty, then $w'$
is equivalent to a word ending in $x^{-1}$, contradicting the choice of $w$.
So $v$ is non-empty and, in the decomposition associated with $U$, $w_{m+1}$
is empty, and $\l{v} = \l{w_m}$ has name different from $x$, so $u_m$
must be critical of type $\{s,t\}$. Furthermore since $vx \in W$ we
must have $m > 1$. The last letter of $w_{m-1}$ cannot lie in $v$,
because that would mean $u_{m-1}$ was critical of type $\{s,t\}$,
$(c,s,t)$, or $(c,t,s)$
for some generator $c$, contradicting condition (iii)
of the optimality of $U$. So the last letter of $w_{m-1}$ is in $w'$. In fact
either this letter is the last letter of $w'$, or $w'$ has a suffix consisting
a prefix of $\alpha_m$.  In either case $w'$ is equivalent to a word ending in
a letter in $\{s,s^{-1},t,t^{-1}\}$, contradicting the choice of $w$.
\end{proof}


\begin{thebibliography}{99}
\bibitem{BCMW} 
R. Blasco-Garcia, M. Cumplido, and R. Morris-Wright,
The word problem is solvable for 3-free Artin groups,
		{\tt arXiv:2204.03523v3}.

\bibitem{BCMWR} R. Boyd, R. Charney, R. Morris-Wright and S. Rees,
	The Artin monoid Cayley graph,
		{\tt arXiv:2303.09504v1}

\bibitem{Chermak} A.~Chermak, Locally non-spherical Artin groups, J. Algebra 
	200 (1998) 56--98.

\bibitem{Deligne} P. Deligne, Les immeubles des groupes de tresses g\'{e}n\'{e}ralis\'{e}s, Inventiones Mathematicae 17 (1972), 273-302.

\bibitem{DehornoyParis} P. Dehornoy and L. Paris, Gaussian groups and Garside groups, two generalisations of Artin groups.
Proc. London Math. Soc. (3)79(1999), no.3, 569–604.


\bibitem{HR12}
D.F. Holt and S. Rees,
Artin groups of large type are shortlex automatic with regular geodesics,
{\em Proc. London Math. Soc.}, 104:486--512, 2012.

\bibitem{HR13}
D.F. Holt and S. Rees,
Shortlex automaticity and geodesic regularity in Artin groups, Groups,
{\em Complex. Cryptol.}, 5:1--23, 2013.

\bibitem{Juhasz} A. Juhasz,
	On the isoperimetric functions of a class of Artin groups, preprint.

\bibitem{MairesseMatheus} J. Mairesse and F. Math\'eus, Growth series for Artin
groups of dihedral type, Int. J. Alg. Comp. 16 (2006) 1087--1107.

\end{thebibliography}
\end{document}